\documentclass[sn-mathphys-num]{sn-jnl}
\usepackage{graphicx}
\usepackage{amsmath}
\usepackage{amssymb}
\usepackage{comment}
\usepackage{algorithm}
\usepackage{algpseudocode}
\usepackage{subcaption}

\newcommand{\real}{\mathbb{R}}
\newcommand{\oreal}{\overline{\mathbb{R}}}

\usepackage{enumerate}
\usepackage{geometry}
\usepackage{amsthm}

\usepackage{float}
\usepackage{amsmath,amssymb,amsfonts}%
\usepackage{amsthm}%
\usepackage{xcolor}%
\usepackage{manyfoot}%
\usepackage{multirow}
\usepackage{thmtools}
\usepackage{hyperref}

\DeclareMathOperator*{\argmax}{arg\,max}
\DeclareMathOperator*{\argmin}{arg\,min}

\newtheorem{definition}{Definition}
\newtheorem{theorem}{Theorem}
\newtheorem{remark}{Remark}
\newtheorem{proposition}{Proposition}

\newtheorem{example}{Example}

\newtheorem{assumption}{Assumption}

\begin{document}

\title{Duality for Non-Convex Composite Functions via the Fenchel–Rockafellar Perturbation Framework\textsuperscript{\textdagger}
}{
}

\author*[1]{\fnm{Vittorio} \sur{Latorre}}\email{vittorio.latorre@unimercatorum.it}
\affil*[1]{\orgdiv{Department of Engineering and Science}, \orgname{Universitas Mercatorum}, \orgaddress{\street{Piazza Mattei 10}, \city{Rome}, \postcode{00186}, \state{Lazio}, \country{Italy}}}

\abstract{
We examine the duality theory for a class of non-convex functions obtained by composing a convex function with a continuous one. Using Fenchel duality, we derive a dual problem that satisfies weak duality under general assumptions.
To better understand this duality, we compare it with classical Lagrange duality by analyzing a related, yet more complex, constrained problem. We demonstrate that the newly derived stationarity conditions are equivalent to the classical Lagrange stationarity conditions for the constrained problem, as expected by the close relationship between Fenchel and Lagrange dualities.

We introduce two non-convex optimization problems and prove strong duality results with their respective duals. The second problem is a constrained optimization problem whose dual is obtained through the concurrent use of the duality theory introduced in this paper and classical Lagrange duality for constrained optimization.
We also report  numerical tests where we solve randomly generated instances of the presented problems using an ad-hoc primal-dual potential reduction interior point method that directly exploits the global optimality conditions established in this paper. The results include a comparison with a well-known conic programming solver applied to the convex duals of the analyzed problems. The interior point method successfully reduces the duality gap close to zero, validating the proposed duality framework.
The theory presented in this paper can be applied to various non-convex problems and serves as a valuable tool in the field of hidden convex optimization.}

\begingroup
  \renewcommand\thefootnote{\fnsymbol{footnote}}
  \footnotetext[2]{Preprint version (October 2025). This manuscript presents the full set of theoretical and numerical results, but it is still undergoing polishing and revision before journal submission.}
\endgroup

\maketitle

\section{Introduction}\label{Introduction}
In this paper, we analyze the properties of the following class of composite functions \( f:C\subseteq\mathbb{R}^n \to \overline{\mathbb{R}} \):
\[
f(x) = (V \circ \Phi)(x) = V(\Phi(x)),
\]
where \( \overline{\mathbb{R}} = [-\infty, \infty] \) denotes the extended real number line. The function \( V \) and the mapping \( \Phi \) satisfy the following conditions:
\begin{itemize}
    \item \( V:\mathbb{R}^m \to \overline{\mathbb{R}} \) is proper, convex and lower semicontinuous (lsc);
    \item \( \Phi:\mathbb{R}^n \to \mathbb{R}^m \) is continuous, and nonlinear.
\end{itemize}
A function is called \emph{proper} if it has a non-empty domain, never attains the value \( -\infty \) in its domain, and is not identically equal to \( +\infty \).

We refer to such composite functions as \emph{Implicit Convex Functions}, a term we introduce in this paper since although $f$ is generally non-convex, it inherits some key convex duality properties, in the sense formalized by Rockafellar’s conjugate duality \cite{rock_con,rock}.
As a matter of fact, we show that if the following optimization problem is considered:
\begin{equation}\label{eq: prim0}
\min_x f(x),
\end{equation}
Where \( f \) is an implicit convex function, the parametric perturbation approach in variational analysis can be used to define a Lagrange function and an associated dual problem. Despite the non-convexity of \( f \), this leads to a primal-dual relationship between points in the primal and dual spaces via subdifferential conditions, in a manner analogous to classical Fenchel conjugate duality. Moreover, weak duality holds between the primal and dual problems, and in certain specific cases, strong duality can be established, leading to global optimality conditions.

Another reason for our interest in implicit convex functions is that the unconstrained optimization problem (\ref{eq: prim0}) is closely related to the following constrained optimization problem:
\begin{equation}\label{eq: primc0}
\begin{array}{cl}
\displaystyle \min_{x, y} & V(y) \\
\text{s.t.} & \Phi(x) = y
\end{array}.
\end{equation}
It is easy to observe that Problem (\ref{eq: primc0}) is generally more computationally expensive than Problem (\ref{eq: prim0}), as handling a non-convex feasible set is typically more challenging than optimizing a non-convex objective function. 
We demonstrate that the optimality conditions of Problem (\ref{eq: prim0}) are equivalent to the Lagrangian optimality conditions of Problem (\ref{eq: primc0}). Consequently, if an efficient method for solving Problem (\ref{eq: prim0}) is found—for instance, by leveraging information from its dual problem—then the obtained solution is also optimal for Problem (\ref{eq: primc0}).
 

Convexity is a highly desirable property in optimization, as it guarantees that any local optimal solution is also globally optimal, thereby significantly simplifying the optimization process. However, convexity arises in only a small subset of real-world optimization problems. 
As a result, even the early researchers sought  to extend convex analysis techniques beyond convex settings, aiming to preserve, to some extent, properties that facilitate efficient solution methods. Such extensions are evident in several examples presented in \cite{rock_con}, including cases that venture into non-convex programming.


In particular, this work can be seen as a further application of the general parametric perturbation framework developed by Rockafellar \cite{rock_con}.  As such, the theoretical results developed in this paper offer a principled approach for analyzing the considered class of non-convex problems within the broader context of conjugate duality.

We begin our analysis by defining the Lagrangian function that incorporates the convex conjugate of \( V \). Using this function, we establish a fundamental conjugate subgradient pair relationship and a Fenchel-type inequality for implicit convex functions, establishing a primal-dual optimal pair in the Fenchel sense even in such non-convex setting. By following the standard procedures in conjugate duality, we define a dual function and then identify a subset of the dual space to define the dual problem. 

After laying the foundations of duality for implicit convex functions, we consider the constrained problem in (\ref{eq: primc0}) and use its Lagrangian function to prove the equivalence between the stationary points of the two problems. This equivalence enables us to identify stationary points of the constrained problem without explicitly handling its non-convex feasible set. Furthermore, we establish a saddle point theorem analogous to that of classical Lagrangian duality.

We then focus on strong duality results for two optimization problems involving implicit convex functions, formulating the duals of both problems as convex quadratic semidefinite programs. 
The second problem features a non-convex objective function over a non-convex feasible set and serves as an example of how duality for implicit convex functions can be applied in a constrained setting.
Finally, we present the numerical experiments conducted on randomly generated instances of the two analyzed problems. To this end, we introduce a primal-dual potential reduction method that leverages the theory and optimality conditions presented in this paper, and compare its performance with that of the conic solver SDPT3 \cite{sdpt1, sdpt2}.

Our main contributions are therefore as follows:
\begin{itemize}
    \item The analysis of implicit convex functions within the Fenchel–Rockafellar framework, identifying a new class of non-convex problems for which meaningful duality relations can be established;
    \item A discussion of the optimality conditions for the newly introduced class of functions, in relation to those of an equivalent constrained formulation addressed via classical Lagrange duality;
    \item Strong duality results for two non-convex optimization problems: the first being unconstrained, and the second constrained, illustrating the applicability of the proposed approach;
    \item Numerical experiments comparing a new primal-dual algorithm with a standard conic solver, validating the theoretical results.

\end{itemize}
We also aim to clarify the relationship between the duality approaches used in \cite{lg14, lg15, lat-gao-chaos, snl19, ls14} and classical conjugate duality theory, in order to highlight new insights into such problems.
Throughout the paper, we rely on foundational results from convex analysis and conjugate duality, particularly those presented in \cite{rock_con, rock}. We also refer to \cite{bvcon} for standard formulations of duality in convex optimization, which serve as a basis for comparison.


The paper is organized as follows: In the next section, we present the general framework for implicit convex functions, defining the associated Lagrangian function and dual problem. We also illustrate the connections between the Problems (\ref{eq: prim0}) and  (\ref{eq: primc0}). In Section 3, we present strong duality results for two specific optimization problems. In Section 4, we describe the strategies used to solve the instances of the problems analyzed in Section 3 and present the numerical results. Finally, in Section 5, we draw the conclusions.

{\bf Notations} Throughout this paper, we use $\mathbb{R}^n$ to denote the $n$-dimensional real space, with $\mathbb{R}^n_+$ indicating the positive orthant in $\mathbb{R}^n$, $\overline{\mathbb{R}} = [-\infty, \infty]$ as the extended real number line, and $\mathbb{R}^{n \times m}$ representing the space of $n \times m$ matrices. 
Given a function $g : \mathbb{R}^n \to \mathbb{R}^m$, we denote its subdifferential as $\partial g$. 
With $\langle \cdot, \cdot \rangle$ we denote the inner product of a vector space. We use $\text{Dom}(f)$ to refer to the domain of the function $f$, and ${\cal R}(\cdot)$ to denote the range of either a function or a matrix. Upper-case letters represent matrices. 
We use $S^n$ to denote the space of $n \times n$ symmetric matrices, and $S_+^n$ to denote the set of $n \times n$ positive-semidefinite matrices. $I_n$ denotes the $n \times n$ identity matrix, and $0_n$ denotes the vector in $\mathbb{R}^n$ composed of all zero entries. 
The expression $A \succeq 0$ means that matrix $A$ is positive-semidefinite, and $A \succeq B$ means that $A - B \succeq 0$. 
$A^\dagger$ indicates the Moore-Penrose generalized inverse of $A$, and $\text{tr}(A)$ denotes the trace of matrix $A$. For a matrix $A \in \mathbb{R}^{n \times n}$, $vec(A)$ denotes the vector in $\mathbb{R}^{n^2}$ obtained by concatenating the columns of $A$.
For $\argmin$ and $\argmax$, we follow the well-known definitions in \cite{rock09}.

\section{Problem Definition and General Properties}
In this section, we first present the well-known definition of the convex conjugate of a function (see, for example, \cite{rock09, rock}), which serves as the foundation for the theoretical results of this paper. 
In the first subsection, we analyze the properties of the minimization problem for an implicit convex function. Through the parametric perturbation approach we are able to introduce the Lagrange function associated with the problem and define the dual problem. Weak duality between the primal and dual  is established by exploiting this well-known framework. Furthermore, we able to prove that the subgradient conjugate relations and the Fenchel-Young equality condition for convex conjugacy hold in the considered case as well.  

In the second subsection, we compare the described duality theory  with the classical Lagrange duality theory for constrained optimization. We consider both problems (\ref{eq: prim0}) and (\ref{eq: primc0}), proving the correspondence between their stationarity conditions. Finally, we provide a Lagrangian saddle point theorem that offers global optimality conditions for a primal solution.

To simplify the notation and align with standard duality theory (as in \cite{rock_con}), we extend the function \( f:C\rightarrow\real \) to all of \( \mathbb{R}^n \) by setting:
\begin{equation}\label{eq: extention}
f(x) := \begin{cases}
V(\Phi(x)) & \text{if } x \in C, \\
+\infty & \text{otherwise}.
\end{cases}
\end{equation}
This convention allows us to write unconstrained minimization problems over \( \mathbb{R}^n \), with the constraint \( x \in C \) being implicitly encoded in the function itself.

We start with the definition of the conjugate of a convex function:
\begin{definition}\label{def: concon}
Let $V:\real^m\rightarrow\oreal$ be a proper, convex and lower semicontinuous (lsc) function with effective domain:
$$
\operatorname{dom}(V) := \{ y \in \mathbb{R}^m \mid V(y) < +\infty \}.
$$
Then its convex conjugate is the function $V^*: D\subseteq \real^m\rightarrow \real$ defined by
$$
V^*(\sigma)=\sup_{y\in \real^m}\{\langle y,\sigma\rangle- V(y) \},
$$
where $\sigma\in D\subseteq \real^m$ is the vector of dual variables and $D$ is the dual domain defined as:
$$
D=\{\sigma\in\real^m :  \sup_{y\in \real^m}\{\langle y,\sigma\rangle- V(y) \}<\infty\}.
$$ 
\end{definition}

In the paper, we only consider the conjugates of convex functions, and therefore Definition \ref{def: concon} corresponds to the definition of the Legendre transform in convex analysis. From the results in the literature \cite{rock}, it follows that under the considered assumptions, the function $V^*$ is also proper, convex and lsc, and  the conjugate of $V^*$ is $V$ itself. The other properties of convex functions and their conjugates that we use in the proofs of this paper are reported in Proposition \ref{the: convprop} of Appendix \ref{ap: proof}.

\subsection{Definition of the Dual Problem and Weak Duality}
We start the analysis with the formal definition of Implicit Convex Functions:

\begin{definition}\label{def: icf}We say that the function $f: \real^n \rightarrow \bar{\mathbb{R}}$ is an {\bf Implicit Convex Function} if it can be expressed as:
\[
f(x) := \begin{cases}
V(\Phi(x)) & \text{if } x \in C, \\
+\infty & \text{otherwise},
\end{cases}
\]
where $C \subseteq \mathbb{R}^n$ is a non-empty convex set, $V: \mathbb{R}^m \rightarrow \bar{\mathbb{R}}$ is  proper, convex, and lsc, and $\Phi: \real^n \rightarrow \mathbb{R}^m$ is a nonlinear continuous mapping.
\end{definition}

We assume \(\Phi\) is nonlinear, because the case of affine or linear compositions is already fully captured by classical conjugate duality theory (see for instance, in Example 2 of Chapter 2 of \cite{rock_con}).  Our aim is to generalize beyond this structure by allowing nonlinear mappings \(\Phi\).
\begin{example}
A simple example of an implicit convex function is the well-known Gaussian function $f(x) = e^{-x^2}$. Even though $f(x)$ is non-convex, it can be written as the composition of $V(y) = e^{-y}$, which is proper, convex, and lsc, with $\Phi(x) = x^2$, which is non-linear and continuous.  
\end{example}
\noindent

We consider the following optimization problem, referred to as the \emph{primal problem}:
\begin{equation}\label{eq: prim}
\min_{x \in \real^n} f(x) = V(\Phi(x)).
\end{equation}

To develop a dual formulation, we follow the parametric perturbation approach of Rockafeller (Chapter 4, \cite{rock_con}) and define a perturbed problem:
\[
F(x, u) := V(\Phi(x) - u), \quad u \in \mathbb{R}^m,
\]
where $u$ is the vector of perturbation parameters. This yields a perturbation function \( F: \real^n \times \mathbb{R}^m \to \bar{\mathbb{R}} \), such that when the perturbation parameter $u$ is set to zero, we recover the original problem:
\[
f(x) = F(x, 0),
\]
which ensures that  \( F \) is a valid perturbation function.

As per chapter 4 in \cite{rock_con}, the associated Lagrangian function with Problem (\ref{eq: prim}) is defined as:
\[
L(x, \sigma) := \inf_{u \in \mathbb{R}^m} \left\{ F(x, u) + \langle u, \sigma \rangle \right\}
= \inf_{u \in \mathbb{R}^m} \left\{ V(\Phi(x) - u) + \langle u, \sigma \rangle \right\},
\]
where $\sigma \in\real^m$ is the dual variable in the dual space $D\subseteq \real^m$ and $\mbox{Dom}(L) = C \times D$.

We set the perturbation parameters to be \( u = \Phi(x) - y \), so that \( y = \Phi(x) - u \), and the infimum becomes:
\begin{equation}\label{eq: lag_proto}
L(x, \sigma) = \inf_{y \in \mathbb{R}^m} \left\{ V(y) + \langle \Phi(x) - y, \sigma \rangle \right\}
= \langle \Phi(x), \sigma \rangle - \sup_{y \in \mathbb{R}^m} \left\{ \langle y, \sigma \rangle - V(y) \right\}.
\end{equation}
The second term in (\ref{eq: lag_proto}) is the Fenchel conjugate of \( V \), so the Lagrangian simplifies to:
\begin{equation}\label{eq: aux}
L(x, \sigma) = \langle \Phi(x), \sigma \rangle - V^*(\sigma).
\end{equation}
As with  function $f$, we extend  $L$ on the entire domain $\real^n\times\real^m$ in the following way:
\[
L(x, \sigma) := 
\begin{cases}
\langle \Phi(x), \sigma \rangle - V^*(\sigma) & \text{if } x \in C,\ \sigma \in D, \\
-\infty & \text{if } x\in C, \sigma\notin D,\\
+\infty & \text{if } x\notin C, \sigma\in D.
\end{cases}
\]
This convention ensures that infeasible variables are penalized in the min–max formulation, while \( L \) remains proper on the product domain \( C \times D\).

\begin{example}
For the Gaussian function $f(x) = e^{-x^2}$, the convex conjugate of the function $V(y) = e^{-y}$ is given by:
$$
V^*(\sigma) = -\sigma (\log(-\sigma) - 1),
$$
where the dual space is $D = \{\sigma \in \real : \sigma \le 0\}$. Therefore, the Lagrangian function is:
$$
L(x, \sigma) = \sigma x^2 + \sigma (\log(-\sigma) - 1).
$$
\end{example}

By the assumptions on \( V \) in Definition \ref{def: icf}, its Fenchel conjugate \( V^* \) is proper, convex, and lower semicontinuous. Moreover, computing \( V^* \) is significantly simpler than in the general non-convex case, since the supremum
\[
\sup_{y \in \mathbb{R}^m} \left\{ \langle y, \sigma \rangle - V(y) \right\}
\]
is well-posed and typically admits a closed-form or a easily computable expression when \( V \) is convex. In addition, since \( V \) is convex and lower semicontinuous, we also have \( V^{**} = V \).
It is important to emphasize that these properties do not follow from general conjugate duality theory, but instead arise from the specific structural assumptions in Definition \ref{def: icf}. This is what makes the class of implicit convex functions particularly amenable to dual analysis, and allows us to establish the following result,  whose proof can be found in Appendix \ref{ap: proof}, generalizing some of the key properties of convex functions and their conjugates to the considered case:

\begin{restatable}{theorem}{fen}\label{the: fen_in}
Let $f: \real^n \rightarrow \oreal$ be an implicit convex function according to Definition \ref{def: icf} and let $\bar{x}\in\real^n$ be a point such that $\Phi(\bar{x})$ is in the effective domain of $V$. Then there exists a point $\bar{\sigma}\in\real^m$ such that $\bar{\sigma}$ is a subgradient of $V$ in $\Phi(\bar {x})$ and the following relations always hold: 

\begin{equation}\label{eq: fenrel1}
\Phi(\bar{x}) \in \partial V^*(\bar{\sigma}) 
\Longleftrightarrow  
\bar{\sigma} \in \partial V(\Phi(\bar{x}))  
\Longleftrightarrow
V(\Phi(\bar{x})) + V^*(\bar{\sigma}) - \langle \Phi(\bar{x}), \bar{\sigma} \rangle = 0,
\end{equation}
Moreover we have $\partial V^*=(\partial V)^{-1}$ and $\partial V=(\partial V^*)^{-1}$.
\end{restatable} 
Theorem \ref{the: fen_in} characterizes a fundamental conjugate subgradient pair relationship between the value $\Phi(\bar{x})$ and the subgradient $\bar{\sigma}$, within the Fenchel duality framework. This relationship is guaranteed by the properness, convexity, and lower semi-continuity of function $V$. 
Furthermore, the relations in (\ref{eq: fenrel1}) are related to the optimality conditions for minimizing the composite function $f(x)=V(\Phi(x))$. Specifically, if $\bar{x}$ is a minimizer of $f(x)$, the condition $0 \in \partial f(\bar{x})$  often implies, depending on the properties of $\Phi$, the existence of a $\bar{\sigma} \in \partial V(\Phi(\bar{x}))$ such that the conditions in (\ref{eq: fenrel1}) are met, establishing a primal-dual optimal pair in the Fenchel sense.

We now define the dual function \( g \) and the associated dual problem.

\begin{definition}\label{def: dual}
The \emph{dual function} is defined as
\[
g(\sigma) := \inf_{x \in C} L(x, \sigma) = \inf_{x \in C} \left\{ \langle \Phi(x), \sigma \rangle - V^*(\sigma) \right\},
\]
\end{definition}

The dual function $g: \mathbb{R}^m \rightarrow \bar{\mathbb{R}}$ is an extended-real-valued function. Its effective domain  is $D$,
 the same effective domain of the Fenchel conjugate $V^*$ as reported in Definition \ref{def: concon}. For $\sigma \notin D$, $g(\sigma)$ takes the value $-\infty$. It is easy to show that weak duality holds between $f$ and $g$:

\begin{theorem}[Weak Duality]\label{th: wd1}
Let \( f \) be an implicit convex function as defined in Definition \ref{def: icf}, and let \( g(\sigma) \) be the dual function defined above. Then for any \( x \in C \), \( \sigma \in D \), it holds that:
\[
g(\sigma) \leq f(x).
\]
\end{theorem}

\begin{proof}
By definition of \( g \), we have:
\[
g(\sigma) = \inf_{x \in C} L(x, \sigma) \leq L(x, \sigma) \quad \forall x \in C.
\]
Also, from the derivation above:
\[
L(x, \sigma) = \inf_u \left\{ V(\Phi(x) - u) + \langle u, \sigma \rangle \right\} \leq V(\Phi(x)) = f(x),
\]
since \( u = 0 \) is a feasible choice. Therefore:
\[
g(\sigma) \leq L(x, \sigma) \leq f(x). \qedhere
\]
\end{proof}

To formulate a well-posed dual problem that leverages specific convexity properties for the primal variable $x$, we consider a restricted dual domain $D^+$ where the Lagrangian $L(x,\sigma)$ exhibits convexity with respect to $x$. Therefore, we define:
$$
D^+ := \left\{ \sigma \in D \mid L(x,\sigma) \text{ is convex in } x \right\}.
$$
The dual problem associated with Problem (\ref{eq: prim}) can then be formulated as:
\begin{equation}\label{eq: dual}
\begin{array}{cl}
\displaystyle \max_\sigma & g(\sigma)\\
\text{s.t.} & \sigma \in D^+\\
\end{array}.
\end{equation}

\begin{remark}
The shape and properties of the set \(D\) depend on the function \(V\), while the shape and properties of the set \(D^+\) depend on the function \(\Phi\). Since \(V(\Phi(x))\) is non-convex, the set \(D^+\) can exhibit the following behaviors:
\begin{itemize}
    \item \textbf{Non-convex:} In this case, the dual problem is non-convex as well. However, weak duality still holds, meaning that solving the dual problem provides a valid lower bound for the primal problem. Despite its non-convexity, the dual problem can still offer insights into the structure of the primal problem and its optimality conditions.
    \item \textbf{Empty:} In general, in convex optimization, if the dual problem is infeasible, then the primal problem may fail to attain a finite optimal solution. The emptiness of \( D^+ \) can indicate that the primal is ill-posed or lacks a well-defined minimum, potentially due to the absence of coercivity or regularity conditions. This is particularly relevant in non-convex settings, where the primal problem may admit local minima, but no global solutions because  it is unbounded from below.
In such cases, practitioners may mistakenly attempt to solve the dual problem and misinterpret the absence of dual feasibility as a numerical failure, rather than a structural property of the primal. Analyzing the feasibility of the dual thus provides important diagnostic information about the soundness of the primal formulation.

\end{itemize}
\end{remark}

\begin{example}
The Lagrangian function associated with the Gaussian function is:
$$
L(x,\sigma)= \sigma x^2+\sigma (\log(-\sigma)-1).
$$
Since  only the quadratic term in the Lagrangian function depends on $x$, the infimum \(\inf_{x\in \real^n} L(x,\sigma)\) is reached at \( x=0 \), and the dual function is:
$$
g(\sigma)= \sigma (\log(-\sigma)-1).
$$
For what regards the dual set $D^+$, the Lagrangian is convex in \(x\) only for  \(\sigma=0\) and therefore $D^+=\{0\}$. It is easy to notice that the value of the dual function in $\sigma=0$ is equal to zero, and this is coherent with the infimum value of $f(x)=e^{-x^2}$ that is also equal to $0$, which is approached as $x\rightarrow \pm \infty$, resulting in:
$$
f(x)\ge g(\sigma), \quad \forall \; x \mbox{ and } \sigma.
$$
If we add a quadratic regularization term \( \beta x^2 \) with \(\beta > 0\)  to the primal problem, a well-known practice in many applications involving the Gaussian function, we obtain:
$$
\min_x f(x)+\beta x^2=e^{-x^2}+\beta x^2.
$$
This regularization term is not part of the implicit convex structure used in the transformation, but it modifies the Lagrangian by directly affecting its convexity with respect to \( x \). Therefore, the Lagrangian function associated with the minimization problem becomes:
$$
L(x,\sigma)= (\sigma+\beta) x^2+\sigma (\log(-\sigma)-1).
$$
By the addition of the quadratic regularization term, the dual feasible set changes to:
$$
D^+ := \{\sigma\in\mathbb{R}: -\beta\le\sigma\le 0\}.
$$
So that the dual problem is:
$$
\begin{array}{cc}
\max &\sigma (\log(-\sigma)-1)\\
\text{s.t.} & -\beta\le\sigma\le 0
\end{array}.
$$
We have the following cases:
\begin{enumerate}
\item If \( \beta \ge 1 \), the primal function attains its minimum at \( x^*=0 \) with \( f(x^*)=1 \), while in the dual problem, the maximum is reached at \( \sigma^*=-1 \) with \( g(\sigma^*)=1 \).
\item If \( 0<\beta<1 \), the primal function has two minimum points: \( x^*_{1/2}=\pm \sqrt{-\log(\beta)} \), with \( f(x_1^*)=f(x_2^*)=\beta(1-\log(\beta)) \). In the dual problem, the maximum is reached at the boundary of the feasible set, with \( \sigma^*=-\beta \) and \( g(\sigma^*)=\beta(1-\log(\beta)) \).
\end{enumerate}

This example demonstrates that:
\begin{itemize}
\item Weak duality holds between the primal and the dual problems; 
\item There is no duality gap between the primal optimal solution and the dual optimal solution when the  quadratic regularization term is introduced.
\end{itemize}
A practical example with $\beta=0.5$ is presented in Figure \ref{fig:dual_primal_plot}. We can see that at the optimum both the primal and the dual functions have the same values showing that strong duality holds.
\end{example}

\begin{figure}[htbp]
    \centering
    \includegraphics[width=1\textwidth]{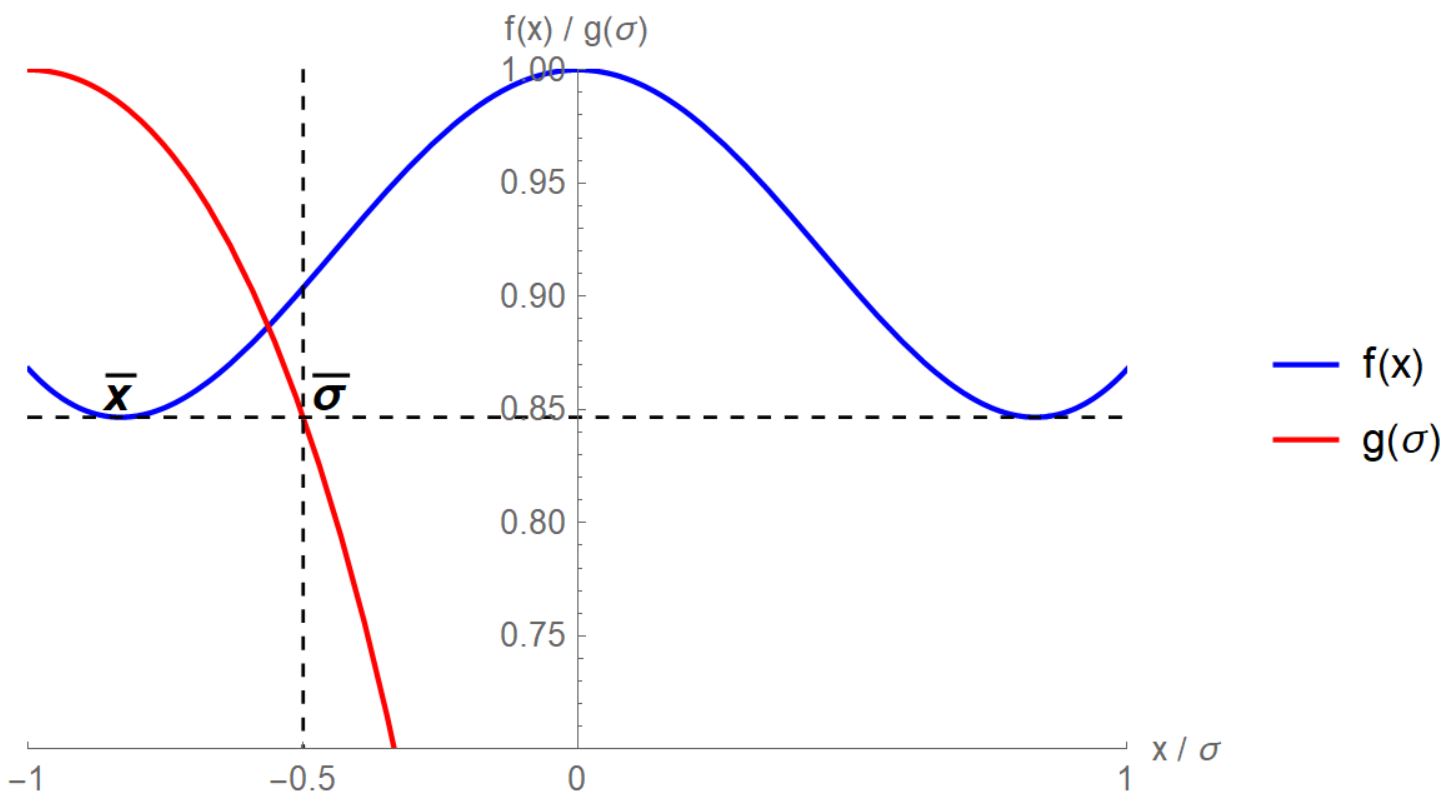} 
    \caption{Comparison between the primal function \( f(x) = e^{-x^2} + \beta x^2 \) and the dual function \( g(\sigma) = \sigma (\log(-\sigma) - 1) \), with \( \beta = 0.5 \). The vertical dashed line at \( \bar{\sigma} = -\beta \) and the horizontal dashed line at \( \beta(1 - \log(\beta)) \) highlight the feasible dual set and common optimal value. The optimal primal points are \( \bar{x} = \pm\sqrt{-\log(\beta)} \).}
    \label{fig:dual_primal_plot}
\end{figure}

\begin{example}
We provide two concrete examples where the dual feasible set \( D^+ \) can be explicitly characterized. These examples illustrate how specific structural assumptions on the mapping \( \Phi \) lead to tractable dual problems:
\begin{enumerate}
\item \textbf{Convex mapping.}  
Let \( \Phi: \mathbb{R}^n \to \mathbb{R}^m \) be a continuous mapping such that each component \( \Phi_i(x) \) is convex for \( i = 1, \dots, m \).  
Then the Lagrangian becomes:
\[
L(x, \sigma) = \sum_{i=1}^m \sigma_i \Phi_i(x) - V^*(\sigma),
\]
in this case, the feasible set for the dual problem is:
$$
D^+=\{\sigma\in\real^m:  \sigma_i \ge 0, i=1,\dots,m\}.
$$
That is the positive orthant, which is a self-dual, closed, and convex cone. This structure frequently appears in convex programming and Linear programming duality.

\item \textbf{Quadratic mapping.}  
Let \( \Phi: \mathbb{R}^n \to \mathbb{R}^m \) be a continuous mapping where each component has the form:
\[
\Phi_i(x) = x^\top A_i x + 2 b_i^\top x + c_i, \quad i = 1, \dots, m,
\]
with symmetric matrices \( A_i \in \mathbb{R}^{n \times n} \), vectors \( b_i \in \mathbb{R}^n \), and scalars \( c_i \in \mathbb{R} \).  
The corresponding Lagrangian is:
\[
L(x, \sigma) = \sum_{i=1}^m \sigma_i (x^\top A_i x + 2 b_i^\top x + c_i) - V^*(\sigma),
\]
and the  feasible set for the dual problem is:
\[
D^+ = \left\{ \sigma \in \mathbb{R}^m : \sum_{i=1}^m \sigma_i A_i \succeq 0 \right\}.
\]
This set is convex and self-dual cone and arises naturally in semidefinite programming (SDP). Its geometry has been extensively studied in the SDP literature.
\end{enumerate}
These cases demonstrate that when $\Phi$ satisfies common structural properties (such as convexity or quadratic form), the dual feasible set $D^+$ admits an explicit and well-understood characterization.
\end{example}

\subsection{Relations with Lagrangian Duality in Constrained Optimization}

After introducing the basics of duality for implicit convex functions, we now explore the relationship between this duality framework and the classical Lagrange duality in constrained optimization.  
To this end, we derive alternative optimality conditions for Problem (\ref{eq: prim}) based on its associated Lagrangian function and prove the equivalence between these conditions and the optimality conditions of the constrained optimization problem (\ref{eq: primc0}), that  we restate here for convenience:
\begin{equation}\label{eq: primc}
\begin{array}{ll}
\displaystyle \min_{x, y} &V(y)\\
\text{s.t.} & \Phi(x)=y
\end{array},
\end{equation}
where \( y \in \mathbb{R}^m \), and the function \( V:\mathbb{R}^m\rightarrow \overline{\mathbb{R}} \) and mapping \( \Phi:\mathbb{R}^n\rightarrow \mathbb{R}^m \) satisfy the same properties as in Definition \ref{def: icf}.  
In addition, we assume in this subsection that \( \Phi \) is locally Lipschitz continuous, in order to apply the subdifferential chain rule to the composite function \( f(x) = V(\Phi(x)) \). This is a standard regularity condition in variational analysis (cf. Section 10 in \cite{rock09}).

We introduce the Lagrangian function associated with Problem (\ref{eq: primc}):
$$
{\cal L}_c(x,y,\sigma) = V(y) - \langle\sigma, ( \Phi(x) - y) \rangle,
$$
where, with a slight abuse of notation, \( \sigma \in \mathbb{R}^m \) denotes the vector of Lagrange multipliers associated with the equality constraint.

To demonstrate the connections between the two problems, we show that Lagrangian function associated with Problem  (\ref{eq: prim})  can also be used to define alternative conditions for the stationary points of \( f \), just as the stationary conditions of  Problem
(\ref{eq: primc}) are obtained through ${\cal L}_c$.
This property is formalized in the following theorem, whose proof can be found in Appendix \ref{ap: proof}.

\begin{restatable}{theorem}{corr}\label{the: corr}
Let \( f(x) = V(\Phi(x)) \) be an implicit convex function with \( \Phi: \mathbb{R}^n \to \mathbb{R}^m \) locally Lipschitz continuous, and let \( L(x,\sigma) := \langle \sigma, \Phi(x) \rangle - V^*(\sigma) \) be its associated Lagrangian.
Let \( \bar{x} \in \mathbb{R}^n \) be a stationary point of \( f \), and let \( \bar{\sigma} \in \mathbb{R}^m \) be such that the Fenchel condition:
\[
\bar{\sigma} \in \partial V(\Phi(\bar{x}))
\]
is satisfied.
Then the pair \( (\bar{x}, \bar{\sigma}) \) is a stationary point of the Lagrangian \( L \), i.e.,
\begin{equation}\label{eq: ld1}
\begin{aligned}
0 &\in \partial_x L(\bar{x}, \bar{\sigma}) \Leftrightarrow 0 \in \partial(\bar{\sigma} \circ \Phi)(\bar{x}), \\
0 &\in \partial_\sigma L(\bar{x}, \bar{\sigma}) \Leftrightarrow 0 \in  \Phi(\bar{x})-\partial V^*(\bar{\sigma}).
\end{aligned}
\end{equation}
Furthermore, the function values coincide:
\begin{equation}\label{eq: valcorr}
f(\bar{x}) = L(\bar{x}, \bar{\sigma}).
\end{equation}
\end{restatable}
\noindent
Here, the expression \( \Phi(\bar{x}) - \partial V^*(\bar{\sigma}) \) denotes a Minkowski set difference:
\[
\Phi(\bar{x}) - \partial V^*(\bar{\sigma}) := \{ \Phi(\bar{x}) - v \mid v \in \partial V^*(\bar{\sigma}) \}.
\]
If a further assumption on the subdifferential is made, we can also prove the following:

\begin{restatable}{corollary}{corrbis}\label{the: corrbis}
Under the same assumptions as in Theorem~\ref{the: corr}, let the pair \( (\bar{x}, \bar{\sigma}) \) be stationary for the Lagrangian \( L \), and assume that the subdifferential \( \partial V(\Phi(\bar{x})) \) is a singleton:
\[
\partial V(\Phi(\bar{x})) = \{ \bar{\sigma} \}.
\]
Then \( \bar{x} \) is a stationary point of \( f \), and the function values coincide:
\[
f(\bar{x}) = L(\bar{x}, \bar{\sigma}).
\]
\end{restatable}

\begin{remark}
The singleton assumption in Corollary \ref{the: corrbis} is automatically satisfied when the function $V$ is differentiable or  strictly convex at the point $\Phi(\bar{x})$. In fact, strict convexity guarantees that the subdifferential $\partial V(\Phi(\bar{x}))$ is a singleton whenever it is nonempty. Under such assumptions, commonly made in practice, especially differentiability, it is possible to establish a one-to-one correspondence between the stationary points of $f$ and those of $L$, and in particular, between their first-order optimality conditions.
\end{remark}

Theorem \ref{the: corr} provides a strong correspondence between the stationary points of \( f \)  those of \( L \), that becomes a one-to-one relation in case of differentiability.  
This connection between Problem (\ref{eq: prim}) and its Lagrangian function can be used to reveal the strong connection between Problems (\ref{eq: prim}) and (\ref{eq: primc}). We show that the stationarity conditions in Theorem \ref{the: corr} are equivalent to the Lagrange stationarity conditions for the  Problem~(\ref{eq: primc}). This equivalence is formally stated below; a complete proof is provided in Appendix~\ref{ap: proof}.

\begin{restatable}{theorem}{kktcon}\label{th: kktcon}
The point $(\bar{x},\bar\sigma)\in\real^{n\times m}$  satisfies the stationarity conditions (\ref{eq: ld1}) for Problem (\ref{eq: prim}) if and only if there exists a $\bar{y}\in\real^m$ such that the point $(\bar{x},\bar{y},\bar\sigma)$ satisfies the Lagrange stationarity conditions of Problem (\ref{eq: primc}):
\begin{subequations}
\begin{align}
0\in&\partial_x {\cal L}_c(\bar{x},\bar{y}, \bar\sigma)=\partial (\bar\sigma \circ \Phi)(\bar{x})   \label{eq: id21}\\
0\in&\partial_y {\cal L}_c(\bar{x},\bar{y}, \bar\sigma)=  \partial V(\bar{y})-\bar\sigma \label{eq: id22}\\
0=&    \Phi(\bar{x})-\bar{y}\label{eq: id23}
\end{align}
\end{subequations}
\end{restatable}

\begin{remark}
Theorem~\ref{th: kktcon} implies that any stationary point of Problem~(\ref{eq: prim}) corresponds to a stationary point of the Problem~(\ref{eq: primc}). In particular, the constraint \(\Phi(x) = y\) is implicitly enforced through the conjugacy conditions arising in the duality framework.

These properties do not generally hold in Fenchel-Rockafellar duality, but are a consequence of the structure of implicit convex functions. Specifically, the conjugacy relations established in Theorem~\ref{the: fen_in} enforce the constraint \(\Phi(x) = y\) at optimality, making it redundant in the primal formulation and allowing equivalence of optimality conditions between the unconstrained and constrained problems.
\end{remark}

Given the equivalence between Problems (\ref{eq: prim}) and (\ref{eq: primc}), one may wonder why Problem (\ref{eq: prim}) is preferred over Problem (\ref{eq: primc}). The primary challenge in solving Problem (\ref{eq: primc}) lies in its feasible region, which is defined by nonlinear equality constraints. This makes it difficult to ensure the feasibility of the sequence of points generated by an optimization algorithm and to fully exploit the convexity of the objective function.  
On the other hand, solving  Problem in (\ref{eq: prim}) is generally easier, even without assuming the convexity of \(V\). However, by leveraging such convexity, implicit convex duality can be used not only to define a dual problem with weak duality properties—providing a lower bound for the primal—but also to derive alternative optimality conditions that bypass the equality constraint.  

Additionally, under strong assumptions, a Lagrange saddle point result can be established, providing global optimality properties.

\begin{restatable}{theorem}{saddle}\label{th: saddle}
Let us consider the primal problem (\ref{eq: prim}), its associated Lagrangian function \(L\), and the dual problem (\ref{eq: dual}). 
Assume that a stationary point \((\bar{x},\bar\sigma)\) of \(L\) is such that $(\bar{x},\bar\sigma)\in C\times D^+$, that is
 \(\bar\sigma \in D^+\). Then:
\[
L(\bar{x}, \sigma) \le L(\bar{x}, \bar\sigma) \le L(x, \bar\sigma), \quad \forall\, x \in C,\, \sigma \in D^+,
\]
that is, the point \((\bar{x}, \bar\sigma)\) is a min-max saddle point of \(L\). Furthermore, \(\bar{x}\) is a global solution to the primal problem (\ref{eq: prim}).
\end{restatable}

In the case where \( D^+ \) is non-empty and convex, the results for saddle functions reported in the literature (such as in references \cite{eke, rock}) can also be applied within this framework. However, in general, there is no guarantee that a stationary point \( (\bar{x}, \bar{\sigma}) \) exists with \( \bar{\sigma} \in D^+ \), since the existence of a saddle point of the Lagrangian on $C\times D^+$ is a necessary and sufficient condition for strong duality, provided the primal problem admits an optimal solution.


\section{Strong Duality Results for Some Non-Convex Optimization Problems}

According to Theorem \ref{th: wd1}, if \( \bar{x} \) is a global solution of problem (\ref{eq: prim}) and \( \bar{\sigma} \) is a global solution of problem (\ref{eq: dual}), the following inequality always holds:
$$
g(\bar{\sigma}) \le f(\bar{x}).
$$
As in classical Lagrange duality for constrained non-convex optimization, a duality gap may exist between the optimal values of the primal and dual problems:
$$
g(\bar{\sigma}) < f(\bar{x}).
$$
In such cases, it is important to note that the solution of the dual problem (\ref{eq: dual}) provides a lower bound on the optimal value of the primal problem and acts as a certificate of quality for the current best solution.
In other instances, we can have strong duality results, which satisfy:
$$
g(\bar{\sigma}) = f(\bar{x}).
$$
While weak duality guarantees that the dual provides a lower bound on the primal solution, it does not ensure that the two values coincide. Without convexity assumptions on the primal problem, strong duality does not  hold in general, but only in specific cases. 
In the remainder of this section, we present two such cases in which strong duality holds for implicit convex functions. The first case is an unconstrained fourth-order polynomial, and the second is the minimization of a fourth-order polynomial with a trust region constraint. In the latter case, both Lagrangian duality for constrained optimization and the duality theory for implicit convex functions are used together.

\subsection{Quartic Polynomial Minimization}
Implicit convex duality can be used to derive the dual and prove strong duality results for the following quartic polynomial minimization problem:
\begin{equation}\label{eq: quar}\tag{$P_{qp}$}
\min_x \,\, f_{qp}(x) = x^T A_0 x + 2b_0^T x + c_0 + \left(x^T A_1 x + 2b_1^T x + c_1\right)^2,
\end{equation}
where \( x \in \mathbb{R}^n \), and the problem parameters are \( A_i \in S^n \), \( b_i \in \mathbb{R}^n \), and \( c_i \in \mathbb{R} \), for \( i = 0, 1 \). This problem is non-convex due to the squared quadratic term, and because the matrices \( A_i \) are not assumed to be positive semidefinite.

Although \( f_{qp} \) is not itself an implicit convex function, it can be decomposed as the sum of a general quadratic term and an implicit convex term:
\[
f_{qp}(x) = \underbrace{x^T A_0 x + 2b_0^T x + c_0}_{\text{quadratic term}} + \underbrace{V(\Phi(x))}_{\text{implicit convex term}},
\]
where
\begin{equation}\label{eq: vquar}
V(y) = y^2, \quad y=\Phi(x) = x^T A_1 x + 2b_1^T x + c_1.
\end{equation}

We apply the duality transformation only to the non-convex term \( V(\Phi(x)) \), which matches the structure of an implicit convex function as defined in Section~2. The remaining quadratic part is added directly to the Lagrangian. Since \( V(y) = y^2 \) is convex, its conjugate is:
\[
V^*(\sigma) = \frac{1}{4} \sigma^2, \quad \text{for } \sigma \in \mathbb{R}.
\]

Applying the general duality framework from Section~2, the Lagrangian function associated with Problem~\eqref{eq: quar} becomes:
\begin{equation}\label{eq: quarl}
L_{qp}(x,\sigma) = x^T A_0 x + 2b_0^T x + c_0 + \sigma \left(x^T A_1 x + 2b_1^T x + c_1 \right) - \frac{1}{4} \sigma^2.
\end{equation}


The dual objective function is easily obtained from (\ref{eq: quarl}), under the standard assumption \cite{bvcon} that $b_0+\sigma b_1 \in {\cal R}(A_0+\sigma A_1)$, where \( \mathcal{R}(\cdot) \) denotes the range of a matrix. First, we use the first-order conditions of \( L_{qp} \) to express \( x \) as a function of \( \sigma \):
$$
x = -(A_0 + \sigma A_1)^\dagger(b_0 + \sigma b_1),
$$
where \( (A_0 + \sigma A_1)^\dagger \) denotes the Moore-Penrose generalized inverse of \( (A_0 + \sigma A_1) \), and substitute this value back into the auxiliary function. Furthermore, \( D^+ \) is determined by noting that from the (\ref{eq: quarl}), \( L_{qp} \) is convex for values of \( \sigma \) such that \( (A_0 + \sigma A_1) \succeq 0 \). Thus, the dual function is:
\begin{equation}\label{eq: landual}
 g_{qp}'(\sigma) =
\begin{cases}
\begin{array}{l}	
c_0 + \sigma c_1  - \frac{1}{4} \sigma^2-\\- (b_0 + \sigma b_1)^T (A_0 + \sigma A_1)^\dagger (b_0 + \sigma b_1)
\end{array} & 
\begin{array}{r}	
(A_0+\sigma A_1)\succeq 0,	\\ b_0+\sigma b_1 \in {\cal R}(A_0+\sigma A_1)
\end{array}\\
-\infty & \quad \text{otherwise.}
\end{cases}
\end{equation}
We observe that the quantity
\[
c_0 + \sigma c_1 - (b_0 + \sigma b_1)^T (A_0 + \sigma A_1)^\dagger (b_0 + \sigma b_1),
\]
matches the standard form of a \emph{Schur complement} inequality. In particular, for a symmetric matrix \( A_0 + \sigma A_1 \succeq 0 \), we can introduce an auxiliary variable \( \gamma \) so that the inequality:
\[
c_0 + \sigma c_1 - \gamma \geq (b_0 + \sigma b_1)^T (A_0 + \sigma A_1)^\dagger (b_0 + \sigma b_1),
\]
is equivalent to the matrix condition:
\[
\left[\begin{array}{cc}
A_0 + \sigma A_1 & b_0 + \sigma b_1 \\
(b_0 + \sigma b_1)^T & c_0 + \sigma c_1 - \gamma
\end{array}\right] \succeq 0,
\]
which allows us to express the dual problem using a linear matrix inequality, leading to the following formulation:
\begin{equation}\label{eq: quard}\tag{$D_{qp}$}
\begin{array}{ll} 
\displaystyle \max_{(\sigma, \gamma)}\,\, g_{qp}(\sigma, \gamma) = & \gamma - \frac{1}{4} \sigma^2 \vspace{0.5em}\\
\displaystyle \text{s.t.} & \displaystyle 
\left[\begin{array}{cc}
A_0 + \sigma A_1 & b_0 + \sigma b_1 \\
(b_0 + \sigma b_1)^T & c_0 + \sigma c_1 - \gamma
\end{array}\right] \succeq 0.
\end{array}
\end{equation}

It is interesting to compare Problem (\ref{eq: quard}) with the dual of one of the most well-known non-convex optimization problems for which strong duality holds: the single constraint quadratic optimization problem:
\begin{equation}\label{eq: quadcon}
\tag{$P_{qc}$}
\begin{array}{cc}
\displaystyle \min_x & \displaystyle f_{qc}(x) = x^T A_0 x + 2 b_0^T x + c_0 \vspace{0.5em}\\
\displaystyle \text{s.t.} & \displaystyle x^T A_1 x + 2 b_1^T x + c_1 \le 0
\end{array},
\end{equation}
with variable \( x \in \mathbb{R}^n \), and problem parameters \( A_i \in S^n \), \( b_i \in \mathbb{R}^n \), and \( c_i \in \mathbb{R} \), for \( i = 0, 1 \). This problem is well-known as the generalized version of the trust region problem in optimization, which has the simpler constraint \( x^T x \le 1 \), and is the principal problem that inspired the research into \emph{Hidden Convex Optimization} \cite{bente96}. The non-convexity of Problem (\ref{eq: quadcon}) arises because the matrices \( A_0 \) and \( A_1 \) are not necessarily positive semi-definite.
\emph{Hidden Convex Optimization} refers to a substantial line of research focused on extending convex optimization techniques to certain classes of non-convex problems. The term is introduced in \cite{bente96}, though foundational ideas appeared much earlier, as surveyed in \cite{horst84}. For a more comprehensive and recent overview, see \cite{hcsur}.


By standard procedures of Lagrange duality in constrained optimization, we obtain the dual of Problem (\ref{eq: quadcon}) as (see Appendix B of \cite{bvcon} and citations therein for more insights):
\begin{equation}\label{eq: quadd} \tag{$D_{qc}$}
\begin{array}{cll} 
\displaystyle \max_{(\lambda, \gamma)} & g_{qc}(\lambda, \gamma) =  \gamma \vspace{0.5em}\\
\displaystyle \text{s.t.} & \lambda \ge 0 \vspace{0.5em}\\
& \displaystyle \left[\begin{array}{cc}
A_0 + \lambda A_1 & b_0 + \lambda b_1 \\
(b_0 + \lambda b_1)^T & c_0 + \lambda c_1 - \gamma
\end{array}\right] \succeq 0
\end{array},
\end{equation}
where \( \lambda \ge 0 \) is the Lagrange multiplier associated with the quadratic constraint.

There are two principal differences between the dual problems:
\begin{enumerate}
\item The dual variable \( \sigma \) in Problem (\ref{eq: quard}) is free in sign, whereas the dual variable \( \lambda \) in Problem (\ref{eq: quadd}) is non-negative.
\item The dual objective function \( g_{qp} \) in Problem (\ref{eq: quard}) contains a quadratic term that corresponds to the conjugate function \( V^*(\sigma) \), which is not present in the dual objective \( g_{qc} \) of Problem (\ref{eq: quadd}).
\end{enumerate}
The absence of the non-negativity constraint in Problem (\ref{eq: quard}) makes it somewhat simpler to analyze. On the other hand, the quadratic term in the objective function of \( g_{qp} \) makes Problem (\ref{eq: quard}) more challenging to handle, though it does not affect its convexity properties.

More importantly, both problems share the same type of positive semi-definite constraint. This characteristic is crucial, as the primary difficulties regarding the strong duality between Problems (\ref{eq: quadcon}) and (\ref{eq: quadd}) stem from this constraint. 

Strong duality does not generally hold for non-convex problems. However, in this case, the structure of the obtained convex dual is similar to that of Problem (\ref{eq: quadd}). Therefore, the techniques to obtain the well-known results of strong duality for the trust region problem can be adapted to the considered case, ensuring the absence of a duality gap.

Before proving such strong duality properties, we introduce the following assumption, which is commonly made in the literature (see Assumption in paragraph 2 of \cite{bente96}) for Problem (\ref{eq: quadcon}) to ensure that the dual feasible space is non-empty:
\begin{assumption}\label{ass: ass1}
There exists \( \alpha \in \mathbb{R} \) such that
$$
(A_0 + \alpha A_1) \succeq 0.
$$
\end{assumption}
It is important to note that Assumption \ref{ass: ass1} must also hold in the case of Problem (\ref{eq: quar}), ensuring that the dual Problem (\ref{eq: quard}) has a non-empty feasible set.
It is easy to verify that under Assumption~\ref{ass: ass1}, the feasible set of problem~\eqref{eq: quard} is convex and non-empty, as it is defined by a linear matrix inequality. Moreover, since the objective function is strictly concave and upper semi-continuous, standard results from convex optimization ensure that the dual problem admits an optimal solution.

We now prove the strong duality result between Problems (\ref{eq: quar}) and (\ref{eq: quard}) in the following theorem:

\begin{theorem}\label{the: sdquad}
Let Assumption \ref{ass: ass1} hold, and let $(\bar\sigma, \bar\gamma)$ be the solution of Problem (\ref{eq: quard}). Then, there exists a point \( \bar{x} \in \mathbb{R}^n \) such that
$$
f_{qp}(\bar{x}) = g_{qp}(\bar\sigma, \bar\gamma),
$$
and \( \bar{x} \) is the global solution of Problem (\ref{eq: quar}).
\end{theorem}

\begin{proof}
If the point $( \bar\sigma,\bar\gamma)$ is a solution of Problem (\ref{eq: quard}), then there exists a matrix of multipliers $\bar{Z}\in S_+^{n+1}$ such that the Karush-Kuhn-Tucker (KKT) conditions of Problem (\ref{eq: quard}) are satisfied:
\begin{subequations}
\begin{align}
 2tr\left(\left[\begin{array}{cc}
 A_1& b_1\\
b_1^T& c_1\\
\end{array}\right]
\bar{Z}\right)=\bar\sigma,&\quad \mbox{Optimality}\vspace{0.5em}\label{eq: kkt1}\\
tr\left(
\left[\begin{array}{cc}
A_0+\bar\sigma A_1&b_0+\bar\sigma b_1\\
(b_0+\bar\sigma b_1)^T&c_0+\bar\sigma c_1-\bar\gamma\\
\end{array}\right]
\bar{Z}\right)=0,&\quad \mbox{Complementarity} \vspace{0.5em}\label{eq: kkt2}\\
\left[\begin{array}{cc}
A_0+\bar\sigma A_1&b_0+\bar\sigma b_1\\
(b_0+\bar\sigma b_1)^T&c_0+\bar\sigma c_1-\bar\gamma\\
\end{array}\right]\succeq0,&\quad \mbox{Dual feasibility}\vspace{0.5em}\label{eq: kkt3}\\
\bar{Z}\succeq 0, \quad \bar{Z}_{n+1,n+1}=1, & \quad \mbox{Multiplier feasibility}\vspace{0.5em}\label{eq: kkt4}
\end{align}
\end{subequations}
where  $\bar{Z}_{n+1,n+1}$ is the $n+1$-th element on the diagonal of matrix $\bar{Z}$. 

A well-known result in the literature on the field of values of symmetric matrices \cite{bri61,hes68}  states that for two symmetric matrices $A,B\in S^{p}$, then for all $Z\in S^p_+$ there exists a vector $z\in\real^p$ such that:
\begin{equation}\label{eq: fov}
z^TAz=tr(AZ),\quad z^TBz=tr(BZ).
\end{equation}
The two matrices we consider are:
\begin{subequations}\label{eq: field}
\begin{align}
tr\left(\left[\begin{array}{cc}
 A_1& b_1\\
b_1^T& c_1-\frac{1}{2}\bar{\sigma}\\
\end{array}\right]\bar{Z}\right) 
=&    
\left[\begin{array}{c}
\bar{v}\\
\bar{w}\\
\end{array}\right]^T\left[\begin{array}{cc}
 A_1& b_1\\
b_1^T& c_1-\frac{1}{2}\bar{\sigma}\\
\end{array}\right]
\left[\begin{array}{c}
\bar{v}\\
\bar{w}\\
\end{array}\right]
,\vspace{0.5em}\label{eq: field1}
\\
tr\left(
\left[\begin{array}{cc}
A_0+\bar\sigma A_1&b_0+\bar\sigma b_1\\
(b_0+\bar\sigma b_1)^T&c_0+\bar\sigma c_1-\bar\gamma\\
\end{array}\right]
\bar{Z}\right)
=&
\left[\begin{array}{c}
\bar{v}\\
\bar{w}\\
\end{array}\right]^T
\left[\begin{array}{cc}
A_0+\bar\sigma A_1&b_0+\bar\sigma b_1\\
(b_0+\bar\sigma b_1)^T&c_0+\bar\sigma c_1-\bar\gamma\\
\end{array}\right]
\left[\begin{array}{c}
\bar{v}\\
\bar{w}\\
\end{array}\right],\label{eq: field2}
\end{align}
\end{subequations}

where $\bar{v}\in \real^n$ and $\bar{w}\in\real$ are the components of the vector $\bar{z}\in\real^{n+1}$ associated with the matrix $\bar{Z}$ through the relation in (\ref{eq: fov}).  
It is easy to notice that both the first and the second relation in the (\ref{eq: field}) are equal to zero. The  (\ref{eq: field1}) because of the (\ref{eq: kkt1}) and the second in the  (\ref{eq: kkt4}), and the (\ref{eq: field2}) because of (\ref{eq: kkt2}).

In the case $\bar{w}\neq 0$, by choosing:
$$
\bar{x}=\bar{v}/\bar{w},
$$
with $\bar{w}\neq 0$, the (\ref{eq: kkt1}) becomes:
\begin{equation}\label{eq: kkt1.1}
\bar{\sigma} = 2\left[\begin{array}{c}
\bar{x}\\
1\\
\end{array}\right]^T\left[\begin{array}{cc}
 A_1& b_1\\
b_1^T& c_1\\
\end{array}\right]
\left[\begin{array}{c}
\bar{x}\\
\bar{1}\\
\end{array}\right].
\end{equation}
We obtain $\bar\gamma$ from the complementarity conditions, as given by equation  (\ref{eq: field2}):
\begin{equation}\label{eq: optg1}
\bar\gamma = \frac{1}{\bar{w}^2}\left[\begin{array}{c}
\bar{v}\\
\bar{w}\\
\end{array}\right]^T
\left[\begin{array}{cc}
 A_0& b_0\\
b_0^T& c_0\\
\end{array}\right]
\left[\begin{array}{c}
\bar{v}\\
\bar{w}\\
\end{array}\right]
+
\frac{\bar\sigma}{\bar{w}^2}\left[\begin{array}{c}
\bar{v}\\
\bar{w}\\
\end{array}\right]^T
\left[\begin{array}{cc}
 A_1& b_1\\
b_1^T& c_1\\
\end{array}\right]
\left[\begin{array}{c}
\bar{v}\\
\bar{w}\\
\end{array}\right],
\end{equation}
from which, by substituting the value of $\bar\sigma$ in the (\ref{eq: kkt1.1}) and $\bar{x}=\bar{v}/\bar{w}$, it results that $\bar\gamma$ is equal to:
\begin{equation}\label{eq: optg2}
\bar\gamma=\left[\begin{array}{c}
\bar{x}\\
1\\
\end{array}\right]^T
\left[\begin{array}{cc}
 A_0& b_0\\
b_0^T& c_0\\
\end{array}\right]
\left[\begin{array}{c}
\bar{x}\\
1\\
\end{array} \right]
+
2\left(\left[\begin{array}{c}
\bar{x}\\
1\\
\end{array} \right]^T
\left[\begin{array}{cc}
 A_1& b_1\\
b_1^T& c_1\\
\end{array}\right]
\left[\begin{array}{c}
\bar{x}\\
1\\
\end{array} \right]
\right)^2.
\end{equation}
By substituting both expressions for \( \bar{\gamma} \) from equation (\ref{eq: optg2}) and \( \bar{\sigma} \) from equation (\ref{eq: kkt1.1}) back into the dual objective function, we obtain:
\begin{equation}\label{eq: objz}
g_{qp}(\bar\sigma,\bar\gamma)=
\left[\begin{array}{c}
\bar{x}\\
1\\
\end{array}\right]^T
\left[\begin{array}{cc}
 A_0& b_0\\
b_0^T& c_0\\
\end{array}\right]
\left[\begin{array}{c}
\bar{x}\\
1\\
\end{array} \right]
+
\left(\left[\begin{array}{c}
\bar{x}\\
1\\
\end{array} \right]^T
\left[\begin{array}{cc}
 A_1& b_1\\
b_1^T& c_1\\
\end{array}\right]
\left[\begin{array}{c}
\bar{x}\\
1\\
\end{array} \right]
\right)^2,
\end{equation}
that is, $g_{qp}(\bar\sigma,\bar\gamma)$ is equal to:
$$
\bar{x}^TA_0\bar{x}+2b_0^T\bar{x}+c_0+(\bar{x}^TA_1\bar{x}+2b_1^T\bar{x}+c_1)^2,
$$
that is exactly the value of the primal problem in the point $\bar{x}$, and therefore:
$$
g_{qp}(\bar\sigma,\bar\gamma)=f_{qp}(\bar{x}),
$$
and, for the results in Theorem \ref{th: wd1}, $\bar{x}$ is the global optimal solution for Problem (\ref{eq: quar}).

In the case $\bar{w}=0$, if $\bar{v}=0_n$, then in the (\ref{eq: optg1}) we encounter the indeterminate form $\frac{0}{0}$, both when we perform the operations $\bar{v}/\bar{w}$ and $\bar{w}/\bar{w}$. 
However, for any positive $\bar{w}=\epsilon>0$, it is possible to obtain $\bar{x}(\epsilon)=\bar{v}/\epsilon=0_n$, $\bar{w}/\bar{w}=\epsilon/\epsilon=1$. In this case we have:
$$
\bar\gamma(\epsilon) = c_0+2 c_1^2, \quad \bar\sigma(\epsilon)=2c_1.
$$
resulting in:
$$
g_{qp}(\bar\sigma(\epsilon),\bar\gamma(\epsilon))=f_{qp}(\bar{x}(\epsilon))=c_0+c_1^2, \quad \forall \epsilon>0.
$$

Now, suppose \( \bar{w} = 0 \) and \( \bar{v} \ne 0_n \). From the KKT conditions (\ref{eq: kkt1}), (\ref{eq: kkt2}) and the matrix identities in (\ref{eq: field}), we deduce that:
$$
\bar{v}^TA_1\bar{v}=0, \quad \bar{v}^TA_0\bar{v}=0.
$$
We fix \( \bar{v} \) and consider a perturbed value \( \hat{w} = \epsilon > 0 \), and analyze the limit as \( \hat{w} \to \bar{w} \), i.e., \( \epsilon \to 0 \). This allows us to approach the case \( \bar{w} = 0 \) while preserving the rank-1 structure in the perturbed matrix \( Z(\epsilon) \). We observe that:
$$
\lim_{\epsilon\rightarrow 0} \frac{\bar{v}^TA_1\bar{v}}{\epsilon^2}=0, \quad  \lim_{\epsilon\rightarrow 0}\frac{\bar{v}^TA_0\bar{v}}{\epsilon^2}=0.
$$
Therefore, from the (\ref{eq: optg1}) we obtain:
$$
\lim_{\epsilon\rightarrow 0}\bar\gamma(\epsilon)=\lim_{\epsilon\rightarrow 0}\left(2b_0^T \frac{\bar{v}}{\epsilon}+c_0\right)+2\left(2b_1^T \frac{\bar{v}}{\epsilon}+c_1\right)^2.
$$
We have four cases:
\begin{enumerate}
\item  ${\bf b_0^T\bar{v}\neq 0, b_1^T\bar{v}\neq 0}$. In this case we have:
$$
\lim_{\epsilon\rightarrow 0}\bar\gamma(\epsilon) = + \infty,
$$
As the term in $b_1^T\bar{v}$ is squared and dominates the first one.
This implies that there exists \( \bar{\epsilon} > 0 \) such that for any \( \epsilon \leq \bar{\epsilon} \), the dual feasibility constraint in equation (\ref{eq: kkt3}) is no longer satisfied. This would make the point \( (\bar{\sigma}, \bar{\gamma}) \) infeasible, contradicting the assumptions.

\item ${\bf b_0^T\bar{v}= 0, b_1^T\bar{v}\neq 0}$. Same as case 1.

\item ${\bf b_0^T\bar{v} \neq 0,\; b_1^T\bar{v} = 0}$. In this case,
\[
\lim_{\epsilon \to 0} \bar\gamma(\epsilon) = \pm \infty.
\]
If \( \bar\gamma(\epsilon) \to +\infty \), we argue as in Case 1. If instead \( \bar\gamma(\epsilon) \to -\infty \), then the objective function tends to \( -\infty \) while \( \bar\sigma = 2c_1 \) remains fixed as consequence of the (\ref{eq: field1}) with $\bar{v}^TA_1\bar{v}=0$, $b_1^T\bar{v} = 0$ and $\hat{w}=\epsilon>0$ .

Since \( \bar\sigma \) is feasible, the Schur complement conditions must be satisfied. In particular, this implies that the vector \( b_0 + \bar\sigma b_1 \) lies in the range of \( A_0 + \bar\sigma A_1 \), i.e., 
\[
b_0 + \bar\sigma b_1 \in \mathcal{R}(A_0 + \bar\sigma A_1).
\]
As a result, we can define a scalar $\tilde\gamma$, so that the point \( (\tilde\gamma, \bar\sigma) \)  is  feasible for Problem (\ref{eq: quard}):
\[
\tilde\gamma = c_0 + \bar\sigma c_1 - (b_0 + \bar\sigma b_1)^T (A_0 + \bar\sigma A_1)^\dagger (b_0 + \bar\sigma b_1),
\]
which is finite due to the above inclusion. Hence, the point \( (\tilde\gamma, \bar\sigma) \) is feasible and achieves a finite objective value. This contradicts the assumption that \( (\bar\gamma, \bar\sigma) \), with \( \bar\gamma(\epsilon) \to -\infty \), is optimal.
\item ${\bf b_0^T\bar{v}= 0, b_1^T\bar{v}= 0}$. We have:
$$
\bar\gamma(\epsilon) = c_0+2 c_1^2, \quad \bar\sigma(\epsilon)=2c_1,
$$
that corresponds with the case $\bar{v}=0_n$ and  $\bar{w}=0$, with no duality gap.
\end{enumerate}
This concludes the proof.
\end{proof}

\subsection{Quadratically-Constrained Quartic Polynomial Minimization}

Since both the duality theory for implicit convex functions and the Lagrangian formulation of constrained optimization problems can be analyzed within the Fenchel-Rockafellar conjugate duality framework, it is natural to investigate whether strong duality properties can be established for constrained problems with implicit convex objectives.

Therefore, in this section, we combine implicit convex duality and Lagrange duality to formulate the dual for the following quartic minimization problem with a quadratic constraint:
\begin{equation}\label{eq: bcqp}\tag{$P_{qq}$}
\begin{array}{cl}
\displaystyle\min_{x \in \mathbb{R}^n} f_{qq} = & \displaystyle (x^T A_1 x + 2 b_1^T x + c_1)^2 \vspace{0.5em} \\
\text{subject to} & \displaystyle x^T A_2 x + 2 b_2^T x + c_2 \leq 0
\end{array},
\end{equation}
where \( A_i \in S^n \), \( b_i \in \mathbb{R}^n \), and \( c_i \in \mathbb{R} \), for \( i = 1, 2 \). We prove the strong duality relations between this problem and its newly formulated dual, demonstrating that the duality results presented in this paper can be seamlessly integrated  with the Lagrange duality theory for constrained optimization to obtain new dual formulations.

First, we report the Lagrangian function associated with Problem (\ref{eq: bcqp}):
\begin{equation}\label{eq: llbq}
\displaystyle {\cal L}_{qq}(x,\lambda)= (x^TA_1x+2b_1^Tx+c_1)^2+ \lambda(x^TA_2x+2b_2^Tx+c_2)
\end{equation}
where $\lambda\in\real_+$ is the Lagrange multiplier associated with the quadratic constraint. Then, we apply the transformation on the fourth-order term in the (\ref{eq: llbq}) and obtain the second Lagrangian function associated with Problem (\ref{eq: bcqp}):
$$
 L_{qq}(x,\lambda,\sigma)=\displaystyle \sigma(x^TA_1x+2b_1^Tx+c_1)+ \lambda (x^TA_2x+2b_2^Tx+c_2)-\frac{1}{4}\sigma^2, 
$$ 
where $\sigma\in\real$ is the dual variable associated with the quartic form in the objective function. 
By defining the set:
\[
C = \left\{ x \in \mathbb{R}^n \;\middle|\; x^T A_1 x + 2b_1^T x + c_1 \le 0 \right\},
\]
it is easy to notice that the following chain of inequalities holds:
\begin{equation}\label{eq: chain}
f_{qq}(x)\ge {\cal L}_{qq}(x,\lambda)\ge L_{qq}(x,\lambda,\sigma) \quad \forall\, x\in C,\,\forall\, \lambda\in\real^+,\,\forall\, \sigma\in\real,
\end{equation}
and the dual function can be obtained as indicated in Definition \ref{def: dual}, i.e., as the infimum of \( L_{qq} \) in \( x \). As in the previous subsection, we can use the first-order conditions of \( L_{qq} \) to express the primal variable as a function of the dual variables:
\begin{equation}\label{eq: xbq}
x=(\sigma A_1 +\lambda A_2 )^\dagger(\sigma b_1+\lambda b_2),
\end{equation}
and notice that the dual feasible set is defined by the constraint: 
$$
\sigma A_1 +\lambda A_2\succeq 0,
$$
and therefore, Assumption \ref{ass: ass1} must be satisfied for this instance as well for the dual set to be non-empty and  have an optimal solution.
Furthermore, we assume that the Slater conditions on the constraint hold, i.e., there exists a \( \hat{x} \) such that:
$$
\hat{x}^TA_2\hat{x}+2b_2^T\hat{x}+c_2<0.
$$
Once we substitute back the value of \( x \) from (\ref{eq: xbq}) into the auxiliary function and use the Schur complement, we obtain  the dual of Problem (\ref{eq: bcqp}) as:
\begin{equation}\label{eq: bcqd}\tag{$D_{qq}$}
\begin{array}{cl}
\displaystyle \max_{\gamma,\lambda,\sigma}  \, \, g_{qq}(\lambda,\sigma,\gamma)=&\gamma-\frac{1}{4}\sigma^2\\
s.t. &\displaystyle \lambda\ge0\\
&
\displaystyle
 \left[\begin{array}{cc}
\sigma A_1 +\lambda A_2 &\sigma b_1+\lambda b_2\\
\left(\sigma b_1+\lambda b_2\right)^T&\sigma c_1+\lambda c_2-\gamma
\end{array}\right]\succeq 0
\end{array}.
\end{equation}
It is easy to notice that, because of the unbroken chain of inequalities in (\ref{eq: chain}), weak duality holds between Problems (\ref{eq: bcqp}) and (\ref{eq: bcqd}). Regarding strong duality, we have the following result:

\begin{theorem}\label{the: sd_qq}
Let Assumption \ref{ass: ass1} and the Slater conditions hold, and let $(\bar\lambda,\bar\sigma,\bar\gamma)$ be a solution of problem (\ref{eq: bcqd}), then there exists a point $\bar{x}\in\real^n$, feasible for Problem (\ref{eq: bcqp}) such that:
$$
f_{qq}(\bar{x})=g_{qq}(\bar\lambda,\bar\sigma,\bar\gamma),
$$
and $\bar{x}$ is the global solution of Problem (\ref{eq: bcqp}).
\end{theorem}
\begin{proof}
By using the result on the field of values of two symmetric matrices, there exist $\bar{v} \in \real^n$ and $\bar{w}\in\real$ such that:
\begin{subequations}\label{eq: field0}
\begin{align}
tr\left(
\left[\begin{array}{cc}
A_1 & b_1\\
b_1^T & c_1-\frac{1}{2}\bar\sigma
\end{array}\right]\bar{Z}\right)
=&
\left[\begin{array}{c}
\bar{v}\\
\bar{w}
\end{array}\right]
\left[\begin{array}{cc}
A_1 & b_1\\
b_1^T & c_1-\frac{1}{2}\bar\sigma
\end{array}\right]
\left[\begin{array}{c}
\bar{v}\\
\bar{w}
\end{array}\right],\label{eq: field01}
\vspace{0.5em}
\\
tr\left(
\left[\begin{array}{cc}
A_2 & b_2\\
b_2^T & c_2
\end{array}\right]\bar{Z}\right)
=&
\left[\begin{array}{c}
\bar{v}\\
\bar{w}
\end{array}\right]
\left[\begin{array}{cc}
A_2 & b_2\\
b_2^T & c_2
\end{array}\right]
\left[\begin{array}{c}
\bar{v}\\
\bar{w}
\end{array}\right],
\label{eq: field02}
\end{align}
\end{subequations}

where $\bar{Z}\in S_+^{n+1}$ is the matrix of multipliers associated with the  dual optimal solution $(\bar\lambda,\bar\sigma,\bar\gamma)$.
Consequently, by considering the vector $\bar{x}=\bar{v}/\bar{w} \in \real^n$, the optimality    and the complementarity conditions on the semidefinite constraint in the dual can be written as:

\begin{subequations}\label{eq: opts1s2}
\begin{align}
2tr\left(
\left[\begin{array}{cc}
A_1 & b_1\\
b_1^T & c_1
\end{array}\right]\bar{Z}\right)
=
2\left[\begin{array}{c}
\bar{x}\\
1
\end{array}\right]
\left[\begin{array}{cc}
A_1 & b_1\\
b_1^T & c_1
\end{array}\right]
\left[\begin{array}{c}
\bar{x}\\
1
\end{array}\right]=&\,\,\bar\sigma
,\label{eq: opts1}
\vspace{0.5em}\\
\left[\begin{array}{c}
\bar{x}\\
1
\end{array}\right]
 \left[\begin{array}{cc}
\bar\sigma A_1 +\bar\lambda A_2 &\bar\sigma b_1+\bar\lambda b_2\\
(\bar\sigma b_1+\bar\lambda b_2)^T&\bar\sigma c_1+\bar\lambda c_2-\bar\gamma
\end{array}\right]
\left[\begin{array}{c}
\bar{x}\\
1
\end{array}\right]=&\,\,0,\label{eq: opts2}
\end{align}
\end{subequations}

To prove the feasibility of $\bar{x}$ in the primal it is sufficient to check the optimality conditions in $\lambda$:
$$
\left[\begin{array}{c}
\bar{x}\\
{1}
\end{array}\right]
\left[\begin{array}{cc}
A_2 & b_2\\
b_2^T & c_2
\end{array}\right]
\left[\begin{array}{c}
\bar{x}\\
{1}
\end{array}\right]+\bar{u}=0,
$$
where $\bar{u} \ge 0$ is the optimal multiplier associated with the constraint $\lambda \ge 0$. Because $\bar{u}\ge0$, we must have that
$$
\left[\begin{array}{c}
\bar{x}\\
{1}
\end{array}\right]
\left[\begin{array}{cc}
A_2 & b_2\\
b_2^T & c_2
\end{array}\right]
\left[\begin{array}{c}
\bar{x}\\
{1}
\end{array}\right]\le0,
$$
That is $\bar{x}$ is feasible for the primal problem. Furthermore  the complementarity condition associated with the positivity of $\lambda$ is such that $\bar{u}\bar{\lambda}=0$. Therefore, in the primal, variable $u$ behaves as a slack variable, and when $\bar\lambda>0$, $\bar{u}$ is forced to be zero. This results in:
$$
\bar\lambda (\bar{x}^TA_2 \bar{x}+2b_2^T \bar{x}+c_2)=0,
$$
that is the complementarity condition for the primal problem holds for $\bar{x}$ and $\bar\lambda$.

To prove that the duality gap is zero, we find the value of $\bar{\gamma}$  by using the complementarity conditions in the  (\ref{eq:  opts2}):
\begin{equation}\label{eq: gammax}
\bar{\gamma}= 
 2\left(\left[\begin{array}{c}
\bar{x}\\
1
\end{array}\right]
\left[\begin{array}{cc}
A_1 & b_1\\
b_1^T & c_1
\end{array}\right]
\left[\begin{array}{c}
\bar{x}\\
1
\end{array}\right]\right)^2
+
\bar\lambda
\left[\begin{array}{c}
\bar{x}\\
1
\end{array}\right]
\left[\begin{array}{cc}
A_2 & b_2\\
b_2^T & c_2
\end{array}\right]
\left[\begin{array}{c}
\bar{x}\\
1
\end{array}\right].
\end{equation}
By substituting back in the dual objective function the value of $\bar\gamma$ in the (\ref{eq: gammax}), together with the value of $\bar\sigma$ in the (\ref{eq: opts1}) it is possible to obtain:
\begin{equation}
g_{qq}(\bar\lambda,\bar\sigma,\bar\gamma) = 
 \left(\left[\begin{array}{c}
\bar{x}\\
1
\end{array}\right]
\left[\begin{array}{cc}
A_1 & b_1\\
b_1^T & c_1
\end{array}\right]
\left[\begin{array}{c}
\bar{x}\\
1
\end{array}\right]\right)^2
+
\bar\lambda
 \left[\begin{array}{c}
\bar{x}\\
1
\end{array}\right]
\left[\begin{array}{cc}
A_2 & b_2\\
b_2^T & c_2
\end{array}\right]
\left[\begin{array}{c}
\bar{x}\\
1
\end{array}\right].\label{eq: dualinprimal}
\end{equation}
It is easy to notice that the second term in the (\ref{eq: dualinprimal}) is zero because the primal complementarity condition holds. Therefore, we have:
$$
f_{qq}(\bar{x})=g_{qq}(\bar\lambda,\bar\sigma,\bar\gamma),
$$
consequently $\bar{x}$ is the global solution of Problem (\ref{eq: quar}).

Finally, in the case \( \bar{w} = 0 \), we observe, by the same  argument used in the proof of Theorem~\ref{the: sdquad}, that this implies \( \bar{v} = 0_n \), and thus \( \bar{x} = \bar{v} / \bar{w} \to 0_n \). As shown in that setting, the corresponding dual and primal values coincide, and \( \bar{x} \) is feasible. Therefore, strong duality holds in this case as well.

\end{proof}

Problem (\ref{eq: bcqp}) combines  the duality of the constrained optimization of problem (\ref{eq: quadcon})  with the duality of problem (\ref{eq: quar}) while still preserving the strong duality results. A further generalization of this problem is:
\begin{equation}\label{eq: gbcqp}
\begin{array}{cc}
\displaystyle\min_{x\in\real^n} &\displaystyle \sum_{i=1}^m (x^TA_ix+2b_i^Tx+c_i)^2\vspace{0.5em}\\
\displaystyle s.t. &\displaystyle x^TAx+2b^Tx+c+\le0
\end{array}.
\end{equation}
Which is a NP-hard problem strongly related to one of the formulations discussed in \cite{nest03}.
By using the  theory of implicit convex functions, it is possible to formulate the convex dual of Problem (\ref{eq: gbcqp}):
\begin{equation}\label{eq: gbcqd}
\begin{array}{ll} 
\displaystyle \max_{(\gamma,\lambda,\sigma )} & \displaystyle\gamma-\frac{1}{4}\sum_{i=1}^m\sigma_i^2 \vspace{0.5em}\\
\displaystyle s.t.&
\displaystyle \left[\begin{array}{cc}
\displaystyle \sum_{i=1}^m\sigma_i A_i+\lambda A&\displaystyle \sum_{i=1}^m\sigma_i b_i++\lambda b\\
\displaystyle\left(\sum_{i=1}^m\sigma_i b_i+\lambda b\right)^T&\displaystyle \sum_{i=1}^m\sigma_i c_i+\lambda c_0-\gamma
\end{array}\right]\succeq 0,
\end{array}
\end{equation}
where $\sigma\in\real^m$.  Even though it is still unknown under which conditions it is possible to prove strong duality, or if such conditions exist at all, the value of the optimal solution of Problem (\ref{eq: gbcqd}) offers a lower bound that can be used as a measurement on the quality of the current best primal solution.

A possible direction to clarify when strong duality holds for this problem can be offered by the generalized S-procedure framework of \cite{volle24}. Their results, based on perturbational duality and minimal convexity assumptions, may help characterize conditions under which the duality gap closes in Problem~(\ref{eq: gbcqd}). A detailed analysis in that direction is left for future work.

\section{Numerical Experience}

To validate our theoretical results, we implement a numerical method for solving the dual problems derived in the previous Sections. Our approach is based on a provably convergent  primal-dual interior-point algorithm, which we compare against a well-known conic programming solver. The primal-dual interior-point method is particularly well suited for this problem as it directly exploits the structure of the primal problem while maintaining feasibility in the dual space.

In the previous section, we showed that the dual problems associated with Problems~\eqref{eq: quar} and~\eqref{eq: bcqp} admit optimal solutions in their respective dual feasible sets, and that strong duality holds in both cases. Moreover, we established the existence of primal-dual pairs satisfying the stationarity conditions of the associated Lagrangian functions.
Therefore, the conditions of Theorem~\ref{th: saddle} are met, and we can conclude that:

\begin{itemize}
\item The Lagrangian \( L_{qp}(x, \sigma) \) has a saddle point \( (x_{qp}^*, \sigma_{qp}^*) \in C \times D^+_{qp} \), where:
$$
D^+_{qp} = \{ \sigma \in \mathbb{R} : A_0 + \sigma A_1 \succeq 0 \},
$$
\item The Lagrangian \( L_{qq}(x, \lambda, \sigma) \) has a saddle point \( (x_{qq}^*, \lambda_{qq}^*, \sigma_{qq}^*) \in C \times D^+_{qq} \), where:
$$
D^+_{qq} = \{ (\lambda, \sigma) \in \mathbb{R}_+ \times \mathbb{R} : \sigma A_1 + \lambda A_2 \succeq 0 \}.
$$
\end{itemize}
In both cases, Theorem~\ref{th: saddle} ensures that these saddle points correspond to globally optimal solutions of the respective primal problems, and can be used to numerically recover the primal solution.

%
In the following subsections, we report a strategy to find such global solutions for the two respective problems and test it on randomly generated instances. In the first subsection, we focus on Problem (\ref{eq: quar}) and present a primal-dual interior-point algorithm for its solution, while in the second subsection, we adapt the algorithm for Problem (\ref{eq: bcqp}).

\subsection{Algorithm Implementation for the Quartic Polynomial Problem}
It is possible to exploit the information from both the primal and dual spaces to solve the following saddle point problem:
$$
\min_{x}\max_{\sigma\in D^+_{qp}} L_{qp}(x,\sigma),
$$
this is equivalent to finding a point  that satisfies  the following conditions \cite{ben05}:
\begin{equation}\label{eq: sadop}
\begin{array}{crl}
\Gamma_{qp}(x,\sigma)=&
\left[
\begin{array}{c}
\nabla_x L_{qp}(x,\sigma)\\
-\nabla_\sigma L_{qp}(x,\sigma)\\
\end{array}
\right]&=0_{n+1}\vspace{0.5em}\\
&A_0+\sigma A_1&\succeq 0
\end{array}.
\end{equation}
Where $\Gamma_{qp}:\real^{n+1}\rightarrow\real^{n+1}$ is monotone. 
 To find such a saddle point, we introduce the matrix of surplus variables $W\in S^n_+$ and the matrix of Lagrange multipliers $U\in S^n_+$ for the semidefinite constraints, 
and reformulate the saddle point problem in (\ref{eq: sadop}) as the problem of finding the zero of the following system of constrained equations \cite{mon99}:
\begin{equation}\label{eq: CE}
\begin{array}{l}
H_{qp}(x,\sigma,U,W)=
\left[
\begin{array}{c}
\Gamma_{qp}(x,\sigma)\vspace{0.5em}\\
vec (A_0+\sigma A_1-W)\vspace{0.5em}\\
vec(WU+UW)/2\\
\end{array}
\right]=0_{n+1+2n^2}\vspace{0.5em}\\
(x,\sigma,U,W)\in \real^{n+1}\times S_+^n\times S_+^n
\end{array}.
\end{equation}
To solve the Problem (\ref{eq: CE}), we take inspiration from the algorithms described in \cite{wright_book,dek},  define the set $\Omega_{qp}=\real^{n+1}\times S_+^n\times S_+^n$, and the following potential reduction function $p_{qp}:\Omega_{qp}\rightarrow \real$:
\begin{equation}
p_{qp}(z)=\rho \log \left(\|H_{qp}(z) \|^2    \right)-\log(Det((WU+UW)/2))\vspace{0.5em}, \quad z=(x,\sigma,U,W)\in\Omega_{qp}\\.
\end{equation}
It is easy to notice that in the potential function we have:
$$
\begin{array}{rrl}
\|vec(A_0+\sigma A_1-W)\|&=&\|A_0+\sigma A_1-W \|_F \\
\|vec(WU+UW)/2\|&=&\|(WU+UW)/2 \|_F,
\end{array}
$$
where \(\| \cdot \|_F\) denotes the Frobenius norm, and \(\rho > n + 1\).

At the $k$-th iteration of the algorithm, the direction used in the optimization process is the Newton direction $\Delta\in\real^{n+1+2n^2}$ obtained by solving the linear system:
\begin{equation}\label{eq: new}
JH_{qp}(z^k)\Delta= -H_{qp}(z^k)+\frac{\beta}{n} o^T H_{qp}(z^k)o,
\end{equation}
where $JH_{qp}$ is the Jacobian of $H_{qp}$, $o=(0_{n^2+n+1},1_{n^2})$, 
 and $1_{n^2}\in\real^{n^2}$ is the vector composed of all ones. The second term on the right-hand side of the (\ref{eq: new}) bends the direction toward the interior of $S_+^n$, so that the product  $(WU+UW)/2$ does not approach the boundary too quickly. This  bending  is controlled by the parameter $\beta\in[0,1)$. 

The matrix $JH_{qp}(z)$ is defined as:
$$
JH_{qp}(z)=\left[
\begin{array} {cccc}
2(A_0+\sigma A_1)& 2(A_1 x+b_1)&0_{n,n^2}&0_{n,n^2}\\
-2(A_0+\sigma A_1)^T&1/2&0_{1,n^2}&0_{1,n^2}\\
0_{n^2,n} & vec(A_1) & 0_{n^2,n^2}& -I_{n^2}\\
0_{n^2,n}&0_{n^2,1}& W\otimes W'& U \otimes U' 
\end{array}
\right]
$$
It is easy to notice that the direction $\Delta$ obtained by solving system (\ref{eq: new}) is a vector in $\real^{n+1+2n^2}$. The last $2n^2$ elements of this direction can be reshaped into two $n\times n$ matrices, $\Delta_U$ and $\Delta_W$, which can be directly added to $U^k$ and $W^k$ when moving to the new point. Because the operator considers the term (WU+UW)/2, the matrices $\Delta_U$ and $\Delta_W$ are symmetric, and therefore the sequences of matrices $\{U^k\}$ and $\{W^k\}$ generated by the algorithm are symmetric as well. Furthermore, in the last \(n^2\) rows of the Jacobian, it is possible to observe the vectorized version of the left-hand side of the Lyapunov equations, which in matrix form are given by:
$$
W^k\Delta_U+ \Delta_U W^k+ \Delta_W U^k+  U^k \Delta_W= -(W^k U^k+ U^k W^k)/2.
$$
\begin{algorithm}
\caption{Primal-Dual Interior-point  (PR)}
\begin{algorithmic}[1]
\State Choose $ (x^0,\sigma^0,U^0,W^0) \in \Omega_{qp} $, $\eta\in (0,1)$, $\epsilon>0$, $M\in\mathbb{N}$ and set $ k = 0$.
\While     {$\|\Gamma_{qp}(x^k,\sigma^k)\|>\epsilon$}
\State Choose a scalar $\beta^k \in [0,1),$ and find a solution $\Delta^k=(\Delta_x^k,\Delta_\sigma^k,\Delta_U^k,\Delta^k_W)$ of the  system of linear equations:
$$
JH_{qp}(z^k)\Delta= -H_{qp}(z^k)+\frac{\beta}{n} o^T H_{qp}(z^k)o
$$
\State Compute the value $p_{max}=\max_{i\in\{k-M,\dots,k\}} p(z^i)$ and find a step size $\alpha>0$ such  that  $z^k+ \alpha \Delta^k\in \Omega_{qp}$ and
$$
p(z^k+\alpha \Delta^k)\le p_{max}+\alpha\eta \nabla p(z^k)^T\Delta^k
$$
and set $\alpha^k=\alpha$\vspace{0.5em}

\State Set  $z^{k+1} =z^k+\alpha^k \Delta^k$, $k = k+1 $.
\EndWhile
\State Return the optimal solution $(x^*,\sigma^*,U^*,W^*)$
\end{algorithmic}\label{alg: pr}
\end{algorithm}
The pseudo code of the primal-dual interior point algorithm is reported in Algorithm \ref{alg: pr}. The algorithm starts from a feasible initial point and stops when the norm of the mapping \(\Gamma_{qp}(x^k, \sigma^k)\) is below a certain threshold \(\epsilon > 0\). We use \(\Gamma_{qp}(x, \sigma)\) instead of the full operator \(H_{qp}(z)\) because it is a more accurate measure of optimality. The Newton direction \(\Delta\) is computed according to equation (\ref{eq: new}), but we employ a dimension-reduction technique to solve it. In fact, the system can be written as:
\[
\begin{array}{rl}
2(A_0 + \sigma A_1)\Delta_x + 2(A_1 x + b_1)\Delta_\sigma &= b_x \\
-2(A_0 + \sigma A_1)^T \Delta_x + \frac{1}{2} \Delta_\sigma &= b_\sigma \\
\text{vec}(A_1) \Delta_\sigma - \Delta_W &= b_W \\
(W \otimes W') \Delta_U + (U \otimes U') \Delta_W &= b_U
\end{array}
\]
It is possible to express \(\Delta_\sigma\) as a function of \(\Delta_x\) and substitute it into the first equation. Once both \(\Delta_x\) and \(\Delta_\sigma\) are computed, the computation of \(\Delta_W\) is straightforward. The last direction \(\Delta_U\) is computed by considering the corresponding Lyapunov system of equations in the variable \(\Delta_U\).
For the line search, we opt for an Armijo-type non-monotone line search \cite{gri86} because we have observed a faster convergence rate for the algorithm. We allow the primal and dual variables to have different step sizes, as the primal variables are not constrained. At line 5, the current point is updated.

Algorithm \ref{alg: pr} is implemented in MATLAB, and the numerical solution for the linear system to compute \(\Delta_x\) is performed using the function \(mldivide\), while the direction \(\Delta_U\) is computed using the function \(lyap\). Both methods benefit significantly from the dimension reduction, which greatly reduces the main computational cost of the algorithm.

The potential reduction algorithm is well-defined as long as the Jacobian is not singular. This occurs if the matrix \(A_0 + \sigma A_1\) and the product \(WU\) are both non-singular. In the algorithm, both the potential function and the line search ensure that the sequence of points generated remains in \(\Omega_{qp}\), which guarantees that these two matrices remain in the interior of \(S_+^n\) and are positive definite. 

The convergence properties of Algorithm \ref{alg: pr} follow from the results in \cite{wright_book,dek} and are similar to those of the algorithm presented in \cite{snl19}. While a full convergence analysis is beyond the scope of this paper, the method leverages established techniques that ensure provable convergence. Moreover, since the algorithm follows an interior-point framework, its time complexity is polynomial.


\begin{table}[ht]
\begin{center}
\begin{small}
\setlength\tabcolsep{2pt}
\begin{tabular}{l|l|cccc}
\hline
n	&	Cond	&	Avg. Gap	&	Avg. Gap	&	AVG. Time	&	AVG. Time	\\
	&		&	 PR	&	 SDPT3	&	PR (sec)	&	SDPT3 (sec)	\\
\hline
100	&	10	&	4.32E-04	&	1.44E-03	&	0.28	&	0.64	\\
	&	100	&	9.39E-08	&	1.27E-05	&	0.26	&	0.58	\\
\hline
200	&	10	&	2.92E-07	&	6.32E-06	&	0.84	&	1.27	\\
	&	100	&	2.27E-06	&	7.54E-06	&	0.87	&	1.31	\\
\hline
300	&	10	&	1.37E-07	&	9.88E-06	&	2.20	&	2.87	\\
	&	100	&	5.82E-06	&	1.26E-05	&	2.25	&	2.92	\\
\hline
400	&	10	&	5.81E-07	&	7.40E-06	&	5.14	&	6.08	\\
	&	100	&	1.30E-05	&	1.40E-05	&	4.73	&	6.07	\\
\hline
500	&	10	&	8.75E-07	&	4.70E-06	&	7.43	&	10.92	\\
	&	100	&	8.83E-07	&	2.14E-05	&	5.99	&	10.90	\\
\hline
\end{tabular}
\end{small}
\end{center}
\caption{Results of the experiments for problem (\ref{eq: quar}), varying the size $n$ of the generated problems and the condition number of the matrices. The reported results are an average over ten runs.}\label{tab: quar}
\end {table}

The potential reduction algorithm is compared with the SDPT3 solver \cite{sdpt1, sdpt2}, which is implemented in C with a MATLAB interface provided by CVX \cite{cvx}. Algorithm SDPT3 solves Problem (\ref{eq: quard}), returning both the dual solution \((\tilde\sigma^*, \tilde\gamma^*)\) and the matrix of multipliers \(\tilde{X}^*\). 

To compare the two methods, we apply the singular value decomposition to the matrix \(\tilde{X}^*\) to find the vector \(\tilde{x}^*\) such that \(\tilde{X}^* = \tilde{x}^*\tilde{x}^{*T}\). Then, we compute the gap as:
\[
gap_{SDPT3} = \left(\tilde{x}^{*T} A_1 \tilde{x}^{*} + 2b_1^T \tilde{x}^{*} + c_1 \right)^2 - \frac{\tilde\sigma^{*2}}{4}
\]
and compare it with the gap obtained from the potential reduction (PR) algorithm, computed as:
\[
gap_{PR} = \left(x^{*T} A_1 x^{*} + 2b_1^T x^{*} + c_1 \right)^2 - \frac{\sigma^{*2}}{4}
\]
We use this gap because it is a measure of optimality derived from equation (\ref{eq: fenrel1}) in Theorem \ref{the: fen_in}. Moreover, this gap does not consider the variable \(\gamma\), which is not computed by Algorithm \ref{alg: pr}, and it is also valid for Problem (\ref{eq: bcqp}).

As previously mentioned, the potential reduction algorithm is implemented in MATLAB, while SDPT3 is written in C with a MATLAB interface. Although the most computationally expensive parts of Algorithm \ref{alg: pr} are performed efficiently using the functions $mldivide$ and $lyap$, MATLAB is an interpreted language and is generally slower than highly optimized C code. Both algorithms  run in a Windows environment on a machine equipped with an Intel i7-11750H processor and 64 GB of RAM, using MATLAB 2021a.

The parameter values used are \(\eta = 10^{-6}\), \(\epsilon = 10^{-4}\), \(\rho = 2(n+1)\), and \(M = 5\). The results are summarized in Table \ref{tab: quar}. For this test, we increased the size of the primal problem \(n\) from 100 to 500. Additionally, we varied the condition number of the generated matrices to assess whether problems with more ill-conditioned matrices are more challenging to solve. The reported results are averages over ten runs with different random number generators. Matrix \(A_0\) is indefinite, while matrix \(A_1\) is positive definite. This choice ensures that Assumption \ref{ass: ass1} is satisfied without the two matrices  being related by simultaneous diagonalization.  

It is evident that both methods are capable of bringing the gap close to zero, indicating that the returned solutions are near the global optimum. When comparing the two algorithms, we observe that the potential reduction primal-dual algorithm is both faster and more effective at reducing the gap than SDPT3. This is likely due to the fact that the potential reduction algorithm directly exploits duality theory, and in the primal space, it only needs to consider a vector of \( \mathbb{R}^n \) variables, as opposed to a matrix of \( \mathbb{R}^{n \times n} \) variables that must be constrained to be positive definite. Furthermore, increasing the conditioning number of the two matrices generally results in lower precision.

\subsection{Algorithm Implementation for the Quadratically-Constrained Quartic Polynomial Problem}
Algorithm \ref{alg: pr} can be easily adapted to solve Problem (\ref{eq: bcqp}). For this problem, we define the operator as follows:
$$
\begin{array}{l}
H_{qq}(x,\lambda,\sigma,w,U,W)=
\left[
\begin{array}{c}
\Gamma_{qp}(x,\lambda,\sigma) \vspace{0.5em} \\
x^T A_2 x + 2 b_2^T x + c_2 + w \vspace{0.5em} \\
\text{vec}(\sigma A_1 + \lambda A_2 - W) \vspace{0.5em} \\
\lambda w \vspace{0.5em} \\
\text{vec}(WU + UW)/{2} \\
\end{array}
\right] = 0_{n+3+2n^2} \vspace{0.5em} \\
(x, \lambda, \sigma, w, U, W) \in \Omega_{qq}
\end{array},
$$
where \(w \geq 0\) is the slack variable for the primal quadratic constraint. The set \(\Omega_{qq}\) is defined as:
$$
\Omega_{qq} = \mathbb{R}^{n+1} \times \mathbb{R}_+ \times \mathbb{R}_+ \times S_+^n \times S_+^n,
$$
and the operator \(\Gamma_{qq}(x,\lambda,\sigma)\) is given by:
$$
\Gamma_{qq}(x,\lambda,\sigma) =
\left[
\begin{array}{c}
\nabla_x L_{qp}(x,\lambda,\sigma) \\
-\nabla_\sigma L_{qp}(x,\lambda,\sigma) \\
\end{array}
\right].
$$
The Jacobian associated with the operator \(H_{qq}\), where \(z = (x, \lambda, \sigma, w, U, W)\), is:
$$
JH_{qq}(z) =
\left[
\begin{array}{cccccc}
2(\sigma A_1 + \lambda A_2) & 2(A_2 x + b_2) & 2(A_1 x + b_1) & 0_{1,n} & 0_{n,n^2} & 0_{n,n^2} \\
2(A_2 x + b_2)^T & 0 & 0 & -1 & 0_{1,n^2} & 0_{1,n^2} \\
-2(A_0 + \sigma A_1)^T & 0 & \frac{1}{2} & 0 & 0_{1,n^2} & 0_{1,n^2} \\
0_{n^2,n} & \text{vec}(A_2) & \text{vec}(A_1) & 0_{n^2,1} & 0_{n^2,n^2} & -I_{n^2} \\
0 & w & 0 & \lambda & 0_{1,n^2} & 0_{1,n^2} \\
0_{n^2,n} & 0_{n^2,1} & 0_{n^2,1} & 0_{n^2,1} & W \otimes W' & U \otimes U'
\end{array}
\right].
$$
The potential function we consider for this problem is similar to \(p_{qp}(z)\) and is defined as:
$$
p_{qq}(z) = \rho \log \left(\| H_{qq}(x, \lambda, \sigma, w, U, W) \|^2 \right) - \log \left( \text{Det} \left( \frac{WU + UW}{2} \right) \right) - \log(\lambda w).
$$

The main difference when  Algorithm \ref{alg: pr} is adapted to this problem is the introduction of a barrier term for the complementarity conditions of the quadratic constraint. Consequently, Algorithm \ref{alg: pr} can be applied to the problem with the necessary adjustments to the operator, the computation of the direction, and the potential function. Additionally, the same dimension reduction technique can be used for computing the Newton direction. Specifically, a linear system of size \(n\) can be solved to compute \(\Delta_x\), and this can be used to determine \(\Delta_\sigma\), \(\Delta_\lambda\), \(\Delta_w\), and \(\Delta_W\). Furthermore, a Lyapunov system can be employed to compute \(\Delta_U\).

\begin{figure}[ht]
    \centering
    \begin{subfigure}[t]{0.48\textwidth}
        \centering
        \includegraphics[width=\textwidth]{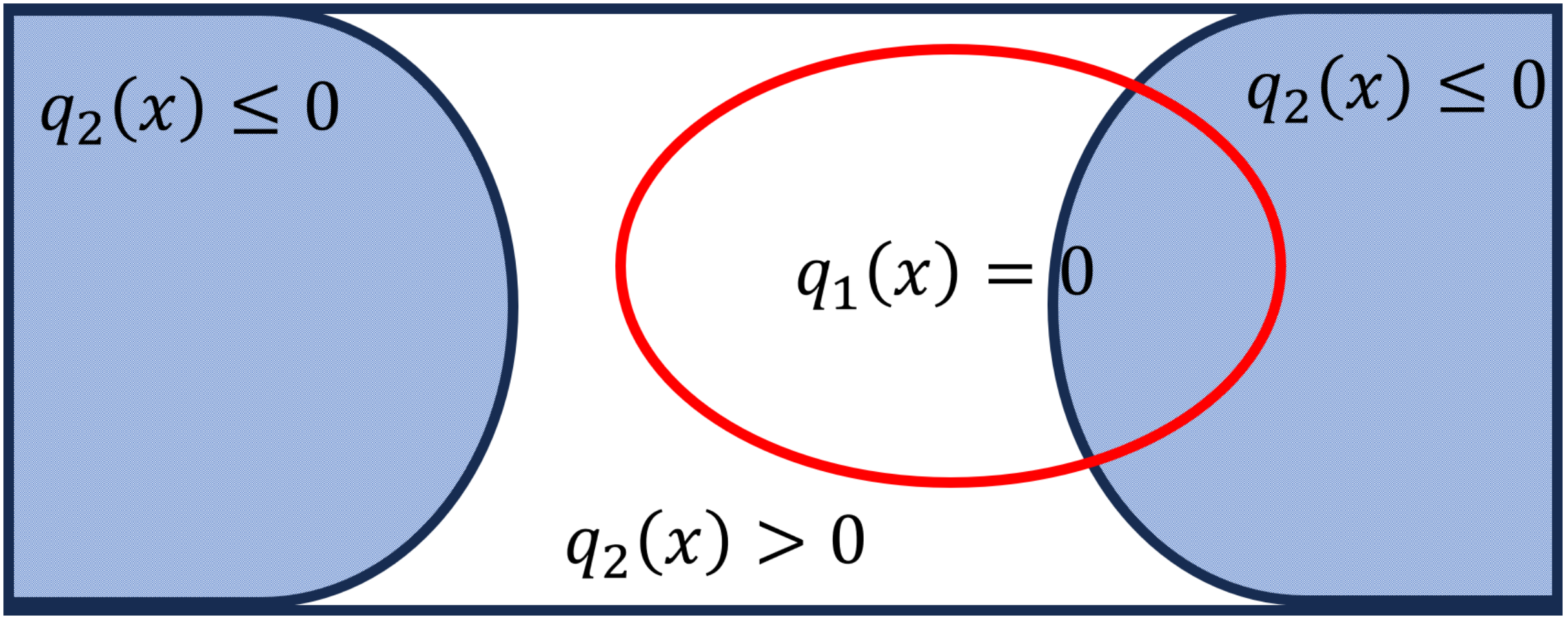} 
        \caption{Two dimensional example of a problem with $A_2$ indefinite.}
        \label{fig:subfig1}
    \end{subfigure}
    \begin{subfigure}[t]{0.48\textwidth}
        \centering
        \includegraphics[width=\textwidth]{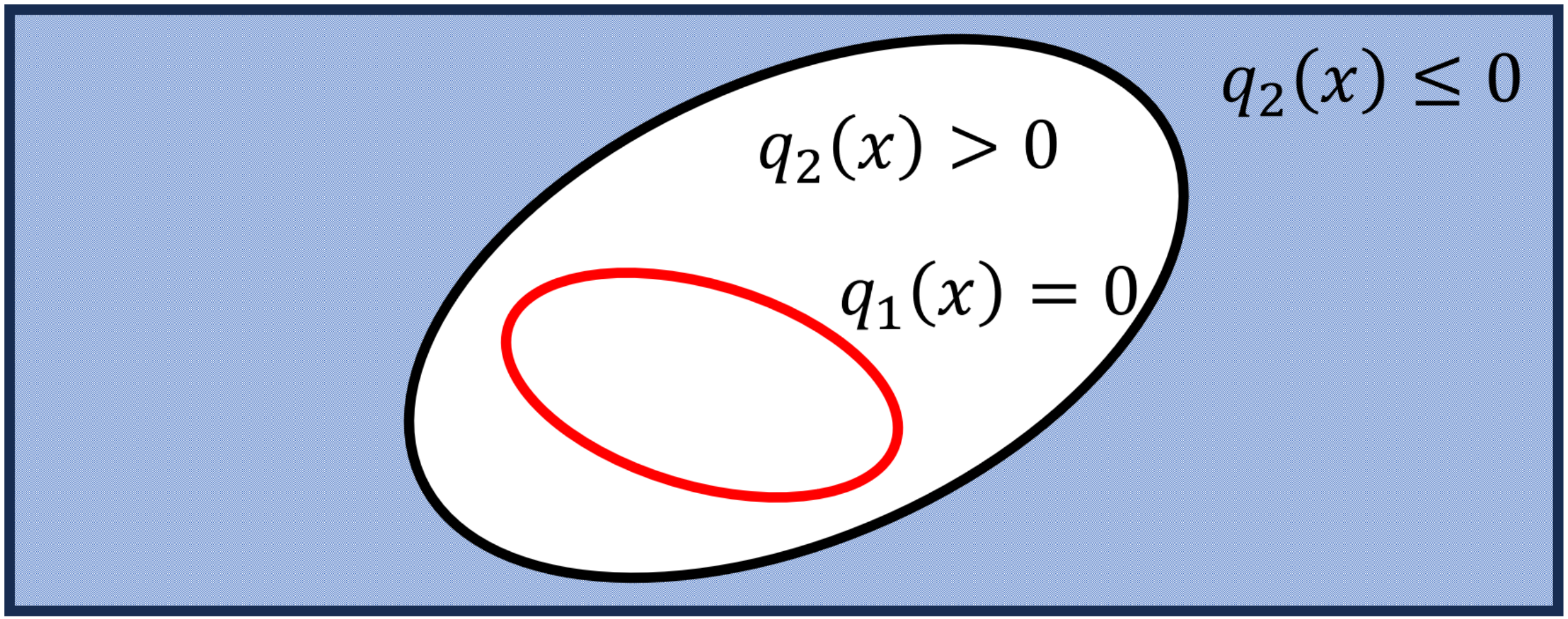} 
        \caption{Two dimensional example with $A_2$ negative definite.}
        \label{fig:subfig2}
    \end{subfigure}
    \caption{Two dimensional examples of the considered test cases for Problem (\ref{eq: bcqp}). The blue filled space represents the feasible set, the red ellipse represents level curve of the objective function such that $q_1(x)=0$.}
    \label{fig:mainfig}
\end{figure}

The tests performed for Problem (\ref{eq: bcqp}) fall into two categories:
\begin{enumerate}
\item The first category consists of \(A_1\) being positive definite and \(A_2\) indefinite;
\item The second category consists of \(A_1\) being positive definite and \(A_2\) negative definite.
\end{enumerate}
This choice is made because, in the case where \(A_1\) is positive definite and \(A_2\) is indefinite, the value of \(f_{qq}\) at the optimum would, in most cases, be zero. Specifically, if we define
$$
q_1(x) = x^T A_1 x + 2 b_1^T x + c_1, \quad q_2(x) = x^T A_2 x + 2 b_2^T x + c_2,
$$
and there exists a point \(\bar{x} \in \mathbb{R}^n\) such that \(q_1(\bar{x}) = 0\) and \(q_2(\bar{x}) < 0\), then this point would be optimal for the problem, as illustrated in Figure \ref{fig:subfig1}, where the points of the red ellipse inside the feasible region are all optimal. This implies that, in general, at the optimum, both the value of the multiplier \(\bar{\lambda}\) and the dual variable \(\bar{\sigma}\) are zero because:
\begin{itemize}
\item The constraint \(q_2(x) \leq 0\) is inactive;
\item The optimality condition for \(\sigma\) is  \(\sigma = 2 q_1(x)\), and $q_1(\bar{x}) = 0$.
\end{itemize}

Therefore, we  consider the second type of tests to examine how the algorithm behaves when both variables do not tend to zero at the optimum. For these tests, we randomly generate the positive definite matrices \(A_0\) and \(A_1\), along with the vectors \(b_0\) and \(b_1\). The values of \(c_0\) and \(c_1\) are then computed such that the level curve \(q_1(x) = 0\) lies inside the ellipse defined by \(x^T A_0 x + 2 b_0^T x + c_0 = 0\). We then set \(A_2 = -A_0\), \(b_2 = -b_0\), and \(c_2 = -c_0\), so that the ellipse \(q_1(x) = 0\) is outside the feasible region, as shown in Figure \ref{fig:subfig2}. In this configuration, not only is the global solution on the boundary of the feasible set, but also the value of \(\sigma^*\) is  nonzero. Consequently, the values of \(\sigma\) must compensate for the values of \(\lambda\) so that \(\sigma A_1 + \lambda A_2 \) remains positive definite during the execution of the algorithm.

\begin{table}[ht]
\small
    \centering
    \begin{subtable}[t]{0.48\textwidth}
        \centering
\setlength\tabcolsep{3pt}
        \begin{tabular}{l|l|ccc}
        \hline
n	&	Cond	&	Average	&	Average dual	&	Time	\\
	&		&	Gap PR	&	KKT PR	&	PR		\\
\hline
100	&	10	&	9.64E-06	&	1.43E-02	&	0.05	\\
	&	100	&	3.46E-09	&	1.75E-03	&	0.03	\\
200	&	10	&	1.51E-06	&	5.67E-02	&	0.09	\\
	&	100	&	6.26E-06	&	7.83E-03	&	0.18	\\
300	&	10	&	1.24E-08	&	1.40E-02	&	0.21	\\
	&	100	&	4.89E-08	&	9.45E-03	&	0.30	\\
400	&	10	&	6.88E-09	&	1.29E-02	&	0.47	\\
	&	100	&	1.05E-06	&	3.08E-03	&	1.33	\\
500	&	10	&	3.16E-06	&	4.95E-02	&	1.08	\\
	&	100	&	5.11E-06	&	5.89E-03	&	1.97	\\
        \hline
        \end{tabular}
        \caption{Experiments for problem (\ref{eq: bcqp}) with $A_2$ indefinite.}
        \label{tab: bcqp2}
    \end{subtable}
    \hfill
    \begin{subtable}[t]{0.48\textwidth}
        \centering
\setlength\tabcolsep{3pt}
        \begin{tabular}{l|l|cccc}
        \hline
 n	&	Cond	&	Average	&	Average dual	&	Time		\\
	&		&	Gap PR	&	KKT PR	&	PR		\\
\hline
100	&	10	&	5.93E-07	&	4.27E-07	&	0.09	&\\
	&	100	&	3.55E-06	&	6.88E-09	&	0.08	&	\\
200	&	10	&	3.75E-02	&	2.90E-02	&	0.40	&	\\
	&	100	&	8.37E-06	&	1.00E-05	&	0.41	&\\
300	&	10	&	1.54E-06	&	4.15E-01	&	1.87	&	\\
	&	100	&	1.06E-05*	&	4.96E-08*	&	2.18	&	\\
400	&	10	&	2.71E-06	&	1.59E-02	&	3.99	&	\\
	&	100	&	1.04E-05*	&	1.96E-06*	&	3.09	&	\\
500	&	10	&	1.32E-06*	&	1.89E-05*	&	5.66	&	\\
	&	100	&	5.85E-06	&	1.30E-03	&	8.33	&	\\
        \hline
        \end{tabular}
        \caption{Experiments for problem (\ref{eq: bcqp}) with $A_2$ negative definite. }
        \label{tab: bcqp1}
    \end{subtable}
    \caption{Comparison of experiments for problem (\ref{eq: bcqp}) under different conditions. The results are the average over ten runs. The asterisk indicates that the PR algorithm failed for some of the instances in that group, for a total of four instances in the second type of tests. Those instances are not considered in the computation of the reported averages.}
    \label{tab:bcqp}
\end{table}

In the tests for the Quadratically-Constrained Quartic Polynomial problem, SDPT3 is unable to return a  solution with satisfactory gap for the majority of instances. We also tested SeDuMi \cite{sed} with the CVX interface, but it similarly fails in most cases. The results from both solvers indicate significant numerical issues. Specifically, as the methods approach optimality, the matrix \(\sigma A_1 + \lambda A_2 \) becomes ill-conditioned, leading to numerical instability that prevents further progress toward closing the duality gap.

It is worth noting that even the proposed primal-dual interior-point method encounters failures in four out of 200 instances, highlighting the inherent difficulty of these problems. However, this failure rate is considerably lower than that observed with standard conic solvers, which frequently fail due to numerical instability.

For these reasons, in Table \ref{tab:bcqp}, we report only the results for the potential reduction algorithm. However, to further assess the quality of the solutions, we include the squared norm of the KKT conditions of the dual problem (\ref{eq: bcqd}), which demonstrates that the returned solution is a good approximation of the dual global optimum.

For these tests, we set \(M = 10\). As expected, the first type of instances is easier to solve than the second, both in terms of precision and computation time. Notably, the second type of test is at least twice as difficult as the first. The squared norm of the KKT conditions for the dual problem remains quite close to zero on average, confirming that the returned values of \(\lambda\) and \(\sigma\) are sufficiently close to the global optimum. Furthermore, all the returned  points satisfy the primal feasibility conditions, even in the instances where the algorithm fails.

We also note that MATLAB occasionally generates warnings about matrix ill-conditioning during the computation of the Newton direction in the potential reduction algorithm when close to optimality. This further highlights the numerical challenges posed by these problems. Nevertheless, the potential reduction algorithm successfully closes the optimality gap in the majority of tested instances.

Finally, the fact that Algorithm \ref{alg: pr} is able to drive both the primal-dual gap and the norm of the KKT conditions of the dual problem close to zero serves as an empirical validation of the theoretical results presented in the previous sections, particularly the strong duality results.


\section{Conclusions}
In this paper, we introduce a duality theory for a class of composite functions, referred to as implicit convex functions. These functions are formed by combining a convex function with a continuous one. By leveraging the property that a modified version of the Fenchel inequality holds for these functions, even in the non-convex case, we define an auxiliary function that is used to formulate the dual problem. From this formulation, weak duality follows naturally. Additionally, we explore the relationship between this duality theory and classical Lagrangian duality in optimization. We first present a constrained problem equivalent to the unconstrained one and then demonstrate how the optimality conditions of both problems are interconnected.

In the second part of the paper, we present strong duality results for two specific instances characterized by implicit convex functions. For the fourth-order polynomial problem, we compare its dual formulation to that of the classic single-constraint quadratic problem, highlighting the differences between these two approaches. In the Quadrically-Constrained Quartic Polynomial problem, we demonstrate how implicit convex duality and Lagrange duality can be combined effectively.

Finally, we introduce a primal-dual interior-point algorithm for the numerical solution of small to medium-sized, randomly generated instances of the two analyzed problems, for which strong duality has been established. We compare this algorithm with a well-established solver. In nearly all instances, the proposed algorithm successfully drives the primal-dual gap close to zero, even when the benchmark solver fails, thereby finding the global solutions for the considered instances and empirically validating the theory presented in this paper.

The technique introduced in this paper offers significant utility for practitioners in several ways. It can be applied to other optimization problems that can be expressed in the implicit convex form, allowing for the formulation of their duals and, in some cases, enabling the proof of strong duality results. As such, the theory of implicit convex functions becomes a valuable tool in the field of hidden convex optimization. From an algorithmic perspective, the dual formulation provides a lower bound for the solution of the current primal incumbent, and in cases where strong duality holds, tailored algorithms can leverage duality to find the global minimum. This opens the door to novel applications and further exploration of new problems and algorithms in optimization.

\section*{Funding}
This research has been jointly funded  by Progetto MolisCTe, Ministero delle Imprese e del Made in Italy, CUP: D33B22000060001 and 
by the Italian Ministry of University and Research (MUR) within the D.M. 10/08/2021 n.1062 (PON Ricerca e Innovazione)

\section*{Data Availability Statement}
The datasets generated during and/or analyzed during the current study are available from the corresponding author on reasonable request.

\section*{Declarations}
\subsection*{Conflict of Interest (COI) Statement}
The author declares that he has no competing interests.

\appendix
\section{Technical Proofs} \label{ap: proof}

We  begin with some the technical properties of convex functions and their conjugates that can be found in Section 23 of \cite{rock} and Proposition 11.3 at page 476 of \cite{rock09}, that are used extensively in the reported proofs:

\begin{proposition}\label{the: convprop}
Let $V: \real^m \rightarrow \oreal $ be a proper, convex and lsc function. Let $V^*$ be its convex conjugate. Then, 
for every \( \bar{y} \in \operatorname{dom}(V) \), 
there exists a point $\bar\sigma\in\real^m$ such that $\bar\sigma$ is a subgradient of V in $\bar{y}$ and the following relations are equivalent:
\begin{equation}\label{eq: fenrel}
\bar{y} \in \partial V^*( \bar{\sigma}) 
\quad \Longleftrightarrow \quad  
 \bar\sigma \in \partial V(\bar{y})  
\quad \Longleftrightarrow \quad
V(\bar{y})+V^*( \bar{\sigma})-\langle \bar{y} ,\bar\sigma\rangle=0.
\end{equation}
Moreover we have $\partial V^*=(\partial V)^{-1}$ and $\partial V=(\partial V^*)^{-1}$.
\end{proposition}

\noindent
\fen*

\begin{proof}

Let $\bar{y} = \Phi(\bar{x})$. By the premise of the theorem, $\bar{y}$ is in the effective domain of $V$. Since $V$ is a proper, convex, and lower semi-continuous function, the subdifferential $\partial V(\bar{y})$ is non-empty. Therefore, there exists a subgradient $\bar{\sigma} \in \partial V(\bar{y})$.  Consequently, the following relations hold from Proposition \ref{the: convprop}:
$$
\bar{y}\in \partial V^*( \bar{\sigma}) 
\Longleftrightarrow  
 \bar\sigma\in \partial V(\bar{y})  
\Longleftrightarrow
V(\bar{y})+V^*( \bar{\sigma})-\langle \bar{y} ,\bar\sigma\rangle=0.
$$
These identities follow from standard duality theory for convex conjugate pairs, based on the Fenchel–Young equality and the Fenchel–Moreau theorem. In our setting, they remain valid due to the assumptions made on implicit convex functions, which ensure that V satisfies \( V^{**} = V \).
Substituting $\bar{y}=\Phi(\bar{x})$ into these equivalent relations yields the desired result:
$$
\Phi(\bar{x})\in \partial V^*( \bar{\sigma}) 
\Longleftrightarrow  
 \bar\sigma\in \partial V(\Phi(\bar{x}))  
\Longleftrightarrow
V(\Phi(\bar{x}))+V^*( \bar{\sigma})-\langle \Phi(\bar{x}) ,\bar\sigma\rangle=0.
$$
This proves the equivalence stated in Equation (\ref{eq: fenrel1}), completing the proof.

\end{proof}

\corr*
\begin{proof}

Let \( \bar{x} \) be a stationary point of \( f(x) = V(\Phi(x)) \), i.e., \( 0 \in \partial f(\bar{x}) \).  
The dual point $\bar{\sigma}$ is such that \( \bar{\sigma} \in \partial V(\Phi(\bar{x})) \), forming a primal-dual Fenchel pair with \( \bar{x} \), as ensured by Theorem~\ref{the: fen_in}.  
Then, by the properties of the Fenchel conjugate:
\[
\bar{\sigma} \in \partial V(\Phi(\bar{x})) \quad \Leftrightarrow \quad \Phi(\bar{x}) \in \partial V^*(\bar{\sigma}),
\]
which implies:
\[
0 \in   \Phi(\bar{x})-\partial V^*(\bar{\sigma}) = \partial_\sigma L(\bar{x}, \bar{\sigma}).
\]

Moreover, by the generalized chain rule for subdifferentials (see Theorem 10.49 in~\cite{rock09}), we have:
\[
\partial f(\bar{x}) \subset \bigcup \left\{ \partial (\sigma \circ \Phi)(\bar{x}) \mid \sigma \in \partial V(\Phi(\bar{x})) \right\}.
\]
Since \( \bar{\sigma} \in \partial V(\Phi(\bar{x})) \), and \( 0 \in \partial f(\bar{x}) \), it follows that:
\[
0 \in \partial(\bar{\sigma} \circ \Phi)(\bar{x}) = \partial_x L(\bar{x}, \bar{\sigma}).
\]

Therefore, both stationarity conditions for the Lagrangian \( L(x, \sigma) \) are satisfied at the point \( (\bar{x}, \bar{\sigma}) \), and we conclude that \( (\bar{x}, \bar{\sigma}) \) is a stationary point of \( L \).

Finally, the (\ref{eq: valcorr}) comes from the third in (\ref{eq: fenrel1}):
$$
V(\Phi(\bar{x})) +V^*(\bar\sigma)-\langle \Phi(\bar{x}),\bar\sigma \rangle =0   \Rightarrow f(\bar{x})-L(\bar{x},\bar\sigma)=0,
$$
Proving the theorem.

\end{proof}
\corrbis*
\begin{proof}
Let $(\bar{x},\bar{\sigma})$ be a stationary point of the auxiliary function $L$. 
It is easy to notice that the second condition in (\ref{eq: ld1}) implies that the point $(\bar{x},\bar\sigma)$ satisfies the condition  $\Phi(\bar{x}) \in \partial V^*(\bar{\sigma})$,
and therefore from Theorem \ref{the: fen_in} we have:
$$
\bar\sigma\in\partial V(\Phi(\bar x)),
$$
by the extended chain rule for subgradients we have that:
$$
\partial f(x)\subseteq\bigcup \left\{\partial (\sigma\circ \Phi)({\bar{x}})| \sigma \in\partial V(\Phi( \bar{x}))\right\}.
$$
Note that we use $\subseteq$ instead of $\subset$ because of the singleton assumption.
By the fact that $\bar\sigma\in\partial V(\Phi(\bar x))$ and that $\bar\sigma$ is the only element in $\partial V(\Phi(\bar x))$, we have
$$
\partial f(\bar{x})= \partial(\bar{\sigma} \circ \Phi)(\bar{x}),
$$
and therefore
$$
 0 \in\partial f(\bar{x}).
$$
proving that $\bar x$ is stationary for $f$. The fact that 
$$
f(\bar{x}) = L(\bar{x}, \bar{\sigma}).
$$
follows from the third condition in in (\ref{eq: fenrel1}).
\end{proof}

\kktcon*
\begin{proof}
We start assuming that the point $(\bar{x},\bar\sigma)\in\real^{n\times m}$  satisfies the conditions (\ref{eq: ld1}).
It is easy to notice that condition (\ref{eq: id21}) is the same as the first condition in (\ref{eq: ld1}). 
From the fact that the point $(\bar{x},\bar\sigma)$ satisfies the conditions in (\ref{eq: ld1}), such point also satisfies  the relations in (\ref{eq: fenrel1}) that we report here for convenience:
\begin{equation}\label{eq: fenrel12}
\Phi(\bar{x}) \in \partial V^*(\bar{\sigma}) 
\Longleftrightarrow  
\bar{\sigma} \in \partial V(\Phi(\bar{x}))  
\Longleftrightarrow
V(\Phi(\bar{x})) + V^*(\bar{\sigma}) - \langle \Phi(\bar{x}), \bar{\sigma} \rangle = 0.
\end{equation}
The point $\bar{y}\in\real^m$ can be easily determined by setting $\bar{y}=\Phi(\bar{x})$, therefore satisfying condition  \ref{eq: id23}. This choice of $\bar{y}$ also makes the second condition in (\ref{eq: ld1}) equivalent to the (\ref{eq: id22}) through this chain of relations:

\begin{align*}
0 \in \Phi(\bar{x}) - \partial V^*(\bar{\sigma})
&\Leftrightarrow \Phi(\bar{x}) \in \partial V^*(\bar{\sigma}) \\
&\Leftrightarrow \bar{\sigma} \in \partial V(\Phi(\bar{x})) \\
&\Leftrightarrow \bar{\sigma} \in \partial V(\bar{y}) \\
&\Leftrightarrow 0 \in \partial V(\bar{y}) - \bar{\sigma}.
\end{align*}
Making the point $(\bar{x},\bar{y},\bar\sigma)$ stationary for ${\cal L}_c$.

For the sufficiency, the first condition in (\ref{eq: ld1}) easily follows, since the two conditions match, while the second in (\ref{eq: ld1}) is satisfied by setting $\Phi(\bar{x})=\bar{y}$ and using         that $\bar\sigma\in\partial V(\bar{y}))\Leftrightarrow \bar{y}\in \partial V^*(\bar\sigma)$. Hence,  $\Phi(\bar{x})\in\partial V^*(\bar\sigma)$, which implies:
$$
0 \in  \Phi(\bar{x})-\partial V^*(\bar{\sigma}) = \partial_\sigma L(\bar{x}, \bar{\sigma}),
$$
Concluding the equivalence.
\end{proof}

\saddle*
\begin{proof}
Because $L$ is convex in $x$ for every $\sigma\in\ D^+$  and in particular for $\bar\sigma$, we have: 
$$
L(\bar{x},\bar\sigma)-L({x},\bar\sigma)\le \langle s,(\bar{x}-x)\rangle \quad \forall x\in C, \forall s\in \partial_x L(\bar{x},\bar\sigma),
$$
and from the assumption that $(\bar{x},\bar\sigma)$ is a stationary point of $L$, we have $0\in \partial_x L(\bar{x},\bar\sigma)$, therefore:
$$
L(\bar{x},\bar\sigma)\le L({x},\bar\sigma) \quad \forall x\in C.
$$
With the same reasoning, considering the concavity of $L$ in $\sigma$ for every $x\in C$, we have:
$$
L(\bar{x}, \sigma) \le L(\bar{x}, \bar\sigma),
$$
that proves that $(\bar{x},\bar\sigma)$ is a saddle point of $L$.

For the global optimality conditions, Theorem~\ref{the: corr} ensures that
\[
f(\bar{x}) = L(\bar{x}, \bar\sigma).
\]
Moreover, for all \( x \in C \), the Fenchel inequality implies
\[
L(x, \bar\sigma) = \langle \Phi(x), \bar\sigma \rangle - V^*(\bar\sigma) \le V(\Phi(x)) = f(x).
\]
Combining the two, we obtain
\[
f(\bar{x}) = L(\bar{x}, \bar\sigma) \le L(x, \bar\sigma) \le f(x), \quad \forall x \in C,
\]
which proves that \( \bar{x} \) is a global minimizer of \( f \) over \( C \).
\end{proof}

\bibliographystyle{plain}
\bibliography{scholar}

\end{document}